\documentclass[twoside,11pt]{article}

\usepackage{amsmath}
\usepackage{amssymb}
\usepackage{mathtools}
\usepackage{amsthm}
\usepackage{microtype}
\usepackage{natbib}
\usepackage[linesnumbered,ruled,vlined]{algorithm2e}
\usepackage{bbm}

\usepackage{fullpage}
\usepackage[colorlinks=true,
            linkcolor=blue,
            urlcolor=blue,
            citecolor=blue]{hyperref}

\SetKwInput{KwInput}{Input}                
\SetKwInput{KwOutput}{Output}              
\SetKwRepeat{KwRepeat}{repeat}{until}

\newcommand{\EE}{\mathbb{E}}
\newcommand{\NN}{\mathbb{N}}
\newcommand{\PP}{\mathbb{P}}
\newcommand{\QQ}{\mathbb{Q}}
\newcommand{\RR}{\mathbb{R}}
\renewcommand{\SS}{\mathbb{S}}

\newcommand{\cA}{\mathcal{A}}
\newcommand{\cF}{\mathcal{F}}
\newcommand{\cL}{\mathcal{L}}
\newcommand{\cM}{\mathcal{M}}
\newcommand{\cN}{\mathcal{N}}
\newcommand{\cO}{\mathcal{O}}
\newcommand{\cP}{\mathcal{P}}
\newcommand{\cT}{\mathcal{T}}
\newcommand{\cU}{\mathcal{U}}
\newcommand{\cX}{\mathcal{X}}
\newcommand{\cY}{\mathcal{Y}}

\newcommand{\cMFloc}{\mathcal{M}_{\mathcal{F}}^{\operatorname{loc}}}
\newcommand{\cMadloc}{\mathcal{M}^{\operatorname{adloc}}}
\newcommand{\cMFadloc}{\mathcal{M}_{\mathcal{F}}^{\operatorname{adloc}}}
\newcommand{\loc}{\operatorname{loc}}
\newcommand{\adloc}{\operatorname{adloc}}

\newcommand{\diam}{\mathrm{diam}}
\newcommand{\LtwoP}{L_2(\mathbb{P}_X)}
\newcommand{\argmin}{\mathop{\mathrm{argmin}}}

\theoremstyle{plain}
\newtheorem{theorem}{Theorem}[section]
\newtheorem{lemma}[theorem]{Lemma}

\newtheorem{proposition}[theorem]{Proposition}
\newtheorem{remark}[theorem]{Remark}
\newtheorem{example}[theorem]{Example}
\newtheorem{definition}[theorem]{Definition}
\numberwithin{equation}{section}

\allowdisplaybreaks

\begin{document}

\begin{center}
{\LARGE Characterizing the minimax rate of nonparametric regression under bounded star-shaped constraints}

{\large
\begin{center}
Akshay Prasadan and Matey Neykov
\end{center}}

{Department of Statistics \& Data Science, Carnegie Mellon University\\
Department of Statistics and Data Science, Northwestern University\\
[2ex]\texttt{aprasada@andrew.cmu.edu}, ~~~ \texttt{mneykov@northwestern.edu}}
\end{center}

\begin{abstract}
    We quantify the minimax rate for a nonparametric regression model over a star-shaped function class $\cF$ with bounded diameter. We obtain a minimax rate of ${\varepsilon^{\ast}}^2\wedge\diam(\cF)^2$ where \[\varepsilon^{\ast} =\sup\{\varepsilon \ge 0:n\varepsilon^2 \le \log \cMFloc(\varepsilon,c)\},\] where $\log \cMFloc(\cdot, c)$ is the local metric entropy of $\cF$, $c$ is some absolute constant scaling down the entropy radius, and our loss function is the squared population $L_2$ distance over our input space $\cX$. In contrast to classical works on the topic \cite[cf.][]{yang1999information}, our results do not require functions in $\cF$ to be uniformly bounded in sup-norm. In fact, we propose a condition that simultaneously generalizes boundedness in sup-norm and the so-called $L$-sub-Gaussian assumption that appears in the prior literature. In addition, we prove that our estimator is adaptive to the true point in the convex-constrained case, and to the best of our knowledge this is the first such estimator in this general setting. This work builds on the Gaussian sequence framework of \cite{neykov2022minimax} using a similar algorithmic scheme to achieve the minimax rate. Our algorithmic rate also applies with sub-Gaussian noise. We illustrate the utility of this theory with examples including multivariate monotone functions, linear functionals over ellipsoids, and Lipschitz classes. 
\end{abstract}
\section{Introduction}
\label{section:nonparam:introduction}

This paper investigates minimax rates for nonparametric regression models under star-shaped constraints. We build on the techniques of \citet{neykov2022minimax} and \citet{prasadan2024informationtheoreticlimitsrobust} with a universal scheme whose error matches the well-known minimax lower bound due to Fano's inequality. The minimax rate turns out to depend closely on the local geometry of the constraint set, in particular, in terms of its local packing numbers. We are able to handle cases where the functions in $\cF$ do not have a uniform sup-norm bound.\footnote{While we can handle function spaces unbounded in sup norm constraints we are always assuming that the function space has a bounded $\LtwoP$ diameter, defined in Section 1.1.} Our algorithmic upper bound also applies when the noise term is sub-Gaussian.

\subsection{Problem Formulation and Notation}

Let $\cX\subseteq\RR^p$ and let $\PP_X$ be some distribution on $\cX$. Let $\cF$ be a function class over $\cX$. We observe  $\{(X_i,Y_i)\}_{i=1}^n$ where for each $i\in\{1,\dots,n\}$,  $X_i$ is independently drawn from $\PP_X$ and $Y_i = \bar{f}(X_i)+\xi_i \in\cY\subseteq\RR$. Here $\bar{f}\in\cF$ denotes the true regression function and the $\xi_i\sim\cN(0,\sigma^2)$ are i.i.d. noise terms for some fixed constant $\sigma>0$ (which does not scale with the sample size). We also consider sub-Gaussian noise, where we say a random variable $Z$ is sub-Gaussian with mean $\mu$ and variance proxy $\tau^2$ if $\EE[\exp(t(Z-\mu)]\le \exp(t^2\tau^2/2)$ for all $t>0$ and write $Z\sim \mathrm{SG}(\mu,\tau^2)$. In the sub-Gaussian setting, we will assume $\xi_i\sim \mathrm{SG}(0,\sigma^2)$ are drawn i.i.d.  

For any $f,g\in\cF$, define the $\LtwoP$-distance between $f$ and $g$ by \[\|f-g\|_{\LtwoP}=\sqrt{\int |f(x)-g(x)|^2\mathrm{d}\PP_X(x)}.\] The $L_p(\PP_X)$-norm can be defined analogously for $p\ge 2$. We write $\cF\subseteq \LtwoP$ if $\int |f(x)|^2\mathrm{d}\PP_X(x)<\infty$ for all $f\in\cF$. We define $d:=\diam(\cF) := \sup\{\|f-g\|_{\LtwoP}:f,g\in\cF\}$; we assume this quantity is finite throughout the paper. Let $B(f,r, \|\cdot\|_{\LtwoP})= B(f,r)=\{g\in\cF:\|f-g\|_{\LtwoP} \le r\}$. Unless stated otherwise, we will always use the $\LtwoP$-norm for balls and abbreviate the notation as $B(f,r)$. We write $B_2^p(\beta,r)$ for the Euclidean $\ell_2$-ball in $\RR^p$ with radius $r$ and center $\beta\in\RR^p$.

The empirical $L_2$-distance given input points $X_1,\dots,X_n\sim \PP_X$ is defined by \[\|f-g\|_{L_2(\PP_n)} = n^{-1/2}\sqrt{\sum_{i=1}^n |f(X_i)-g(X_i)|^2}.\] We introduce the notation $\vec{Y} = (Y_1,\dots,Y_n), \vec{X} = (X_1,\dots,X_n)$. We denote by $g(\vec{X})$ the vector with $g(X_i)$ in the $i$th component, and similarly for $f(\vec{X})$ and $\bar{f}(\vec{X}).$ We can equivalently write our $L_2(\PP_n)$ distances as $\|f-g\|_{L_2(\PP_n)} = n^{-1/2}\|f(\vec{X})-g(\vec{X})\|_2$, using the Euclidean $L_2$ norm for vectors in $\RR^n$.  

We often abbreviate $\{1,\dots,n\}$ with $[n]$. We write $x\lesssim y$ (or $y\gtrsim x$) if there is some positive absolute constant $\alpha$ such that $x\le \alpha y$. If $x\lesssim y$ and $y\lesssim x$, we write $x\asymp y$. We denote the $p\times p$ identity matrix by $\mathbb{I}_p$.

\begin{definition}
    We say a function class $\cF$ is star-shaped if there exists an $f_c\in\cF$, called the center of $\cF$, such that for any $f\in\cF$ and $t\in[0,1]$, $tf+(1-t)f_c\in \cF$. 
\end{definition}

Henceforth we will assume $\cF\subseteq \LtwoP$ is a star-shaped and complete class of real-valued functions on $\cX\subseteq\RR^p$. Our goal is to  quantify the minimax rate \[\inf_{\hat{f}}\sup_{\bar{f}\in\cF}\EE_{\vec{X},\vec{Y}}\|\hat{f}(\vec{X},\vec{Y})-\bar{f}\|_{\LtwoP}^2\] and achieve it with a universal scheme depending only on the local geometry of $\cF$. Here the infimum is over all measurable functions of the data $\hat{f}\colon \cX^n\times\cY^n\to \cF$.

A crucial assumption we impose on $\cF$ and $\PP_X$ is the following: for some fixed constant $L>0$ and every integer $p\ge 2$ we have \begin{align} \label{new:bernstein:sufficient:condition}
   \|f - g\|_{L_{2p}(\PP_X)}^{2p} \leq \|f - g\|_{\LtwoP}^2 L^{p-1} p!/2.
\end{align} This will later yield us a Bernstein-like concentration inequality from which we can derive our upper bound. The requirement is implied by having a bound on the sup-norm of $\cF$, or an $L$-sub-Gaussian condition (explained in the Related Works section), and is in fact strictly more general than either of these assumptions.

An essential tool will be the local metric entropy of $\cF$, which we now define. Observe that this quantity may be infinite if $\cF$ is not totally bounded, but our theorems will still hold (see Remark \ref{remark:total:boundedness}). 

\begin{definition}[Local Metric Entropy]\label{local:entropy:def:nonparam} We say a set $S$ is an $r$-packing set in the $\LtwoP$ norm for a function class $\mathcal{G}$ defined on $\cX$ if $S\subseteq \mathcal{G}$ and for all distinct $f,g\in S$, $\|f-g\|_{\LtwoP} > r$. We denote by $\cM(r,\mathcal{G})$ the maximal cardinality of an $r$-packing set in the $\LtwoP$ norm of  $\mathcal{G}$. Given our function class $\mathcal{F}$, we define for a constant $c>1$ the following: \begin{align*}
    \cMFloc(\varepsilon, c) = \sup_{f \in \cF} \cM(\varepsilon/c, B(f, \varepsilon) \cap \cF).
\end{align*} Then $\log \cMFloc(\varepsilon, c)$ is the local metric entropy of $\cF$. We may omit the subscript $\cF$ unless necessary.
\end{definition}

The local metric entropy has been used in several existing works to obtain minimax rates for non-parametric regression, including \citet[Section 7]{yang1999information}, \citet{agglomeration_yang_2004}, \citet[Section C.3.]{wang_yang_2014}, and \citet{mendelson_2017_local_versus_global}.

\subsection{Related Works}

The literature on nonparametric regression is vast \citep[e.g.,][]{tsybakov2009introduction}, so we do not give an exhaustive summary. We do however highlight those works closest in spirit to our results.

In \citet{yang1999information}, the authors study minimax rates in terms of global entropies for density estimation, but apply their results to nonparametric regression. They obtain a similar result in their Theorem 6, where the minimax rate is $\varepsilon^2$ for $\varepsilon$ satisfying $\log \cM(\varepsilon,\cF)\asymp n\varepsilon^2$. However, they require $\sup_{f\in \cF}\|f\|_{\infty}<L$ for some $L>0$, while our results relax this assumption. As a consequence, we are able to derive minimax rates on linear functionals with unbounded sup norms. In addition, \cite{yang1999information} impose an assumption on the parameter space which essentially forces the local and global metric entropies to coincide. We do not require this condition in our work.

\citet{mendelson_2017_local_versus_global} considers prediction problems for convex function classes with bounded $L_2(\mathbb{P}_X)$ diameter, including ones with unbounded sup-norm. The minimax rate we consider can be viewed as a special case of this prediction setting, although we have a much broader set of constraints. However, Mendelson bounds the error with high probability rather than in expectation. Moreover, \citet[Definition 2.1]{mendelson_2017_local_versus_global} requires an assumption that $\cF$ is $L$-sub-Gaussian. This definition has two components. First, for each $p\ge 2$ and $f,g\in\cF\cup\{0\}$, there must exist $L>0$ such that $\|f-g\|_{L_p(\PP_X)} \le L\sqrt{p}\|f-g\|_{\LtwoP}$. Second, the centered canonical Gaussian process induced by $\cF$ must be bounded. That is, to each $f\in\cF$, we associate a Gaussian random variable $G_f$ with $\EE[G_f]=0$ and $\EE[G_f G_g]= \EE[fg]$ for all $f,g\in\cF$, and the assumption is that $\sup_f G_f$ is a bounded random variable almost surely. We impose a strictly weaker condition, which is implied, for example, by an $L$-sub-Gaussian assumption or a bounded sup-norm, but neither of those two conditions are necessary. We do not impose an assumption about boundedness of the Gaussian process. Another distinction with our approach is that \citet{mendelson_2017_local_versus_global} does not precisely match the lower and upper bounds on the error rate. In addition, since every convex set is star-shaped, we generalize the shape constraints.

Other works include \citet{agglomeration_yang_2004}, which examines taking data-dependent convex combinations of $n^{\tau}$ for $\tau>0$ many candidate regression procedures to obtain a rate of order of the minimax rate plus a penalty term. This work requires the function to be bounded in sup-norm, but relaxes assumptions on the distribution of the error term. \citet{yang_2014_high_dimensional_nonparam} studies the nonparametric regression setting under assumptions of sparsity or additivity of the function class. They also require assumptions that the distribution on the input space $\cX=[0,1]^p$ has a density lower bounded on a sub-hypercube (i.e., the data is not from a strictly lower dimensional subspace). The paper also uses global metric entropy conditions related to the minimax rate, as in \citet{yang1999information}. \citet{zhao2023minimaxratesconvergencenonparametric} study a very general set of non-parametric problems allowing for more general noise distributions as well as the noise varying with $X$. Their results impose a bounded sup-norm assumption on the function class, however, in contrast to our work. Moreover, we are interested in scenarios where the richness condition in the authors' Definition 3 may not hold, i.e., when the local and global entropies do not match for small $\varepsilon$.

The techniques in this paper are similar to \citet{neykov2022minimax} in that we successively build finer local packings of the function class, resulting in an infinite tree of points that densely cover the set. We then traverse this tree using  observed data. However, we must perform an additional pruning step to correct a gap found in the argument of \citet{neykov2022minimax}. The procedure is described in full detail in \citet{prasadan2024informationtheoreticlimitsrobust}, where the authors derive minimax rates for adversarially contaminated, sub-Gaussian mean location models with a star-shaped constraint.  A drawback of these techniques is that the tree construction as well as traversal is not efficiently computable, at least in its stated form.

\subsection{Organization}

In Section \ref{section:nonparam:upper:bound}, we give a summary of well-known results about lower bounds for the minimax rates. Section \ref{section:nonparam:upper:bound} establishes a matching upper bound in three steps: Section 3.1 introduces our local packing algorithm for $\mathcal{F}$; Section 3.2 derives useful properties of our key assumption \eqref{new:bernstein:sufficient:condition} (with comparisons to $L$-sub-Gaussianity); Section 3.3 bounds our algorithm's error and proves it attains the minimax rate. Then in Section \ref{section:nonparam:adaptive}, we  prove that the error-rate of the algorithm's output is adaptive to the true $\bar{f}$, albeit only when $\cF$ is convex. 

We then consider in Section \ref{section:nonparam:examples} numerous applications of our theory. One of the simplest examples is the set of linear functionals: given a star-shaped and compact set $K\subset\RR^p$, we consider functions of the form $X^T\beta$ for $\beta\in K$ and draw data $X_1,\dots,X_n$ from an isotropic sub-Gaussian distribution. For $K$ being the $\ell_1$-ball and assuming the true $\beta^{\ast}$ is $s$-sparse, we use Sudakov minoration along with properties of the Gaussian width of tangent cones stated in \citet{chandrasekaran2012convex} and \citet{amelunxen2014living} to derive the known minimax rate of order $\varepsilon\asymp\sqrt{s\log(d/s)/n}$, matching results from \citet{raskutti2011minimax} (see also \citet{wang_yang_2014} for more general bounds). Following the ellipse example in \citet{neykov2022minimax}, we again consider sub-Gaussian design and linear functionals over ellipsoids and recover a novel rate. We compare this to the work of \citet{pathak2023noisy}. We then derive known minimax rates for functions with smoothness assumptions, a well-studied setting \citep{stone_1980, stone_1982}. Lastly, we consider the example of monotone functions in general dimensions. We use existing results about the metric entropy from \citet{yang1999information} and \citet{gao_Wellner_monotone} to recover the well-known minimax rate. We then analyze monotone functions that are piecewise constant on $k$ hyper-rectangles and show that our algorithm matches the adaptation rate of the least squares estimator established by \citet{isotonic_general_dimensions}. 

\section{Lower Bound}
\label{section:nonparam:lower:bound}

Our lower bound results all use existing and well-known information theoretic arguments, but we provide the details for completeness. The same argument using the local entropy appears in \citet[page 1594]{yang1999information}, \citet[Section 6]{agglomeration_yang_2004}, and \citet[Section C.3]{wang_yang_2014}, for example.
The upshot is that the minimax rate is at least $\varepsilon^2$ up to constants where $\varepsilon$ satisfies $\log \cMFloc(\varepsilon, c) > 4(n\varepsilon^2/(2\sigma^2) \vee \log 2)$.

We start by restating Fano's Inequality:
\begin{lemma}[Fano's Inequality for Nonparametric Regression]\label{lemma:fano:nonparam} Let $f_1, \ldots, f_m\in\cF$ be a collection of $\varepsilon$-separated functions, i.e., satisfying $\|f_i-f_j\|_{\LtwoP} > \varepsilon$ for each $i\ne j$. Suppose $J$ is uniformly distributed over the index set $[m]$, and $(Y_i |X_i, J = j) = f_j(X_i) + \xi_i$ where $\xi_i \sim \cN(0,\sigma^2)$ for each $i\in[n],j\in[m]$. Then 
\begin{align*}
    \inf_{\hat{f}}\sup_{\bar{f}\in\cF}\EE_{\vec{X},\vec{Y}}\|\hat{f}(\vec{X},\vec{Y})-\bar{f}\|_{\LtwoP}^2\geq \frac{\varepsilon^2}{4}\bigg(1 - \frac{I(\vec{X},\vec{Y};J) + \log 2}{\log m}\bigg),
\end{align*} where $I(\vec{X},\vec{Y};J)$ is the mutual information between our data and $J$.
\end{lemma}

The next lemma gives an upper bound on the mutual information term that appears in Fano's Inequality. 

\begin{lemma}[Bounding the Mutual Information in Nonparametric Setting] \label{lemma:nonparam:mutualinfo} Under the same set-up as in Lemma \ref{lemma:fano:nonparam}, we have $I(\vec{X},\vec{Y};J) \le \frac{n}{2\sigma^2}\max_j \|f_j-h\|_{\LtwoP}^2$ for any $h\in\cF$.
\end{lemma}

To complete the lower bound, we return to the local metric entropy defined earlier.

\begin{lemma} \label{lemma:nonparam:lowerbound} We have
\begin{align*}
     \inf_{\hat{f}}\sup_{\bar{f}\in\cF}\EE_{\vec{X},\vec{Y}}\|\hat{f}(\vec{X},\vec{Y})-\bar{f}\|_{\LtwoP}^2 \geq \frac{\varepsilon^2}{8 c^2}
\end{align*}
for any $\varepsilon$ satisfying $\log \cMFloc(\varepsilon, c) > 4\left(\frac{n\varepsilon^2}{2\sigma^2} \vee \log 2\right)$, where $c$ is the constant from the Definition \ref{local:entropy:def:nonparam} fixed to some large enough value.
\end{lemma}

\section{Upper Bound}
\label{section:nonparam:upper:bound}

We now establish our upper bound, simultaneously handling the cases where $\cF$ is bounded or unbounded in sup-norm, while always assuming the $\LtwoP$-diameter is finite. More precisely, the bounded sup-norm scenario means there is some $r>0$ such that $\|f\|_\infty \le r$ for all $f\in\cF$. This implies $d\le 2 r$ by the triangle inequality. The unbounded sup-norm assumption means for all $r>0$, there is an $f\in\cF$ and $x\in \cX$ such that $|f(x)|>r$). We implicitly assume $\cF$ is totally-bounded when using the local metric entropy, but Remark \ref{remark:total:boundedness} details how this is not a necessary assumption.

To achieve the optimal rate, we first construct an algorithm that outputs an estimator achieving a squared error rate of $\varepsilon^2$ when $\varepsilon$ satisfies a (discretized) condition of the form $n\varepsilon^2 > 2\log \cMFloc\left(\overline{C}\varepsilon, c\right)\vee \log 2$ for some constant $\overline{C}$. 

To tie together this upper bound with the lower bound from Lemma \ref{lemma:nonparam:lowerbound}, we show that either the minimax rate is $d^2$, in which case our algorithm clearly outputs a function in $\cF$ achieving this error rate, or it has an error rate ${\varepsilon^{\ast}}^2$ satisfying both the lower and upper bound conditions mentioned above. In fact, we establish the rate is ${\varepsilon^{\ast}}^2\wedge d^2$, where \begin{equation} \label{definition:epsilon:star}
    \varepsilon^{\ast}:=\sup\{\varepsilon\ge 0:n\varepsilon^2 \le \log \cMFloc(\varepsilon, c)\}
\end{equation} for some sufficiently large constant $c$.

Throughout this section, we take $C=2(c+1)$ and assume $c$ is sufficiently large.

\subsection{Packing and Pruning a Tree}

Before applying our algorithm, we must perform an intricate construction of iteratively formed local packing sets, as described in \citet{neykov2022minimax}. This produces a directed tree whose nodes densely pack and cover $\cF$, at finer distances the deeper down the tree. However, as is done in \citet{prasadan2024informationtheoreticlimitsrobust}, a critical pruning step is required. Once this pruned tree is constructed, we use our data to traverse the tree (equivalently, updating our estimate of $\bar f$), forming a Cauchy sequence in $\cF$. Due to our use of functions, we use a different criterion to update our estimate of $\bar{f}$, namely,  minimizing the empirical least squares error among functions in the local packing at each step. Moreover, the tree itself is completely independent of the observed data and only depends on the function class.

We direct the reader to \citet[Section 3.1]{prasadan2024informationtheoreticlimitsrobust}, where the tree construction is explained in detail for an arbitrary star-shaped set $K$ with a bounded diameter. We will give a very brief synopsis and write an analogous version for our star-shaped function class.

The tree, denoted $G$, will be constructed level by level, with nodes corresponding to functions in $\cF$. We will denote the $j$th level of $G$ as $\cL(j)$, where the first level is simply the root node. We write $\cP(f)$ as the parent set of a node $f$, meaning there is an edge from each node in $\cP(f)$ to $f$. The offspring set $\cO(f)$ are the set of nodes with a directed node pointing to them from $f$. Note that all notions of distance here are with respect to $\LtwoP$, but we show in Lemma \ref{lemma:equivalent:norms} that any equivalent norm may be used.

To construct $G$, first pick a root node $f_R$, forming $\cL(1)$. Form $\cL(2)$ by taking the points of a $d/c$-maximal packing of $B(f_R,d)\cap\cF$, directing an edge from $f_R$ to each of these points. For $\cL(3)$, at each point $g$ of $\cL(2)$, form a $d/(4c)$-packing of $B(g ,d/2)\cap \cF$, and draw a directed edge from $g$ to these each of those points. 

Now begins our first pruning step. Order the nodes lexicographically (justified by Remark \ref{remark:breaking_ties}), say $f_1^3,f_2^3,\dots,f_{N_3}^3$. Let $\cU_3$ be a list of unprocessed nodes, initially set to $[f_1^3,f_2^3,\dots,f_{N_3}^3]$. Now construct $\cT_3(f_1^3)$ to be the set of nodes in $\cU_3$ (other than $f_1^3$) within distance $d/(4c)$ of $f_1^3$. If empty, update $\cU_3$ by removing $f_1^3$ from it, and proceed to the steps in the following paragraph. Otherwise do the following. For each element $f_j^3$ of $\cT_3(f_1^3)$, add a directed edge from the parent node $\cP(f_j^3)$ to $f_1^3$ and delete $f_j^3$. In other words, if one child node has a `cousin' node that is too close, that child node gets a second parent, and the cousin node and the edge to it from its parent gets deleted.  Once all elements of $\cT_3(f_1^3)$ have been handled, update $\cU_3$ by removing  $\{f_1^3\}\cup \cT_3(f_1^3)$. 

Then we repeat: pick the remaining node in line from $\cU_3$, say $f_k^3$, construct $\cT_3(\cdot)$ as the set of nodes in $\cU_3$ within distance $d/(4c)$ of $f_k^4$, swap edges as needed and delete nodes, then update $\cU_3$. Once $\cU_3$ is empty, the pruning is done. The remaining nodes comprise level 3 of the tree, $\cL(3)$ of $G$ .

The steps for $\cL(4)$ and beyond are identical, except we decrease the radius by a factor of 2 at each level. That is, we form a $d/(8c)$-packing of $B(g,d/4)\cap\cF$ for each $g\in\cL(3)$, form the lexicographically ordered $\cU_4$ storing unprocessed nodes, write $\cT_4(\cdot)$ be functions in $\cU_4$ within distance $d/(8c)$ of the input, swap edges and delete nodes, update $\cU_4$, and repeat. The formal details are given in Algorithm \ref{algorithm:nonparam:directed:tree}, which is essentially identical to Algorithm 1 in \citet{prasadan2024informationtheoreticlimitsrobust}.

\begin{algorithm} 
\SetKwComment{Comment}{/* }{ */}
\caption{Directed Tree Construction \label{algorithm:nonparam:directed:tree}}
\KwInput{Root node $f_R \in \cF$}
$\cL(1)=\{f_R\}$\;
Draw directed edge from $f_R$ to the functions in maximal $(d/c)$-packing of $B(f_R, d)\cap \cF$.
Set $\cL(2)=\cO(f_R)$\;
$k \gets 3$\;
\While{TRUE} {
   For each $f\in \cL(k-1)$, draw directed edge from $u$ to nodes of a maximal $\tfrac{d}{2^{k-1}c}$-packing of $B(f,\tfrac{d}{2^{k-2}})\cap \cF$\;
   $\cU_k$ is a lexicographically ordered list of nodes just added\;
   \While{$\cU_k\ne\varnothing$} {
    Pick first element, say $f_l^k$, of $\cU_k$\;
    Set $\cT_k(f_l^k)=\{f_j^k \in\cU_k:\|f_j^k-f_l^k\|_{\LtwoP}\leq \tfrac{d}{2^{k-1} c}, j\ne l\}$\;
    For each $f_j^k\in\cT_k(f_l^k)$, remove directed edge  $\cP(f_j^k)$ to $f_j^k$ and instead draw an edge from $\cP(f_j^k)$ to $f_l^k$. Delete node $f_j^k$ from $G$\;
    Remove $f_l^k\cup \cT_k(u_l^k)$ from $\cU_k$\;
   }
   Set $\cL(k)=\bigcup_{f\in \cL(k-1)}\cO(f)$\;
   $k\gets k+1$\;
}
\Return{$G$} 
\end{algorithm}

Some basic lemmas for this directed tree construction are proven in \citet{prasadan2024informationtheoreticlimitsrobust}, albeit not necessarily for a function class. We restate the results (with slightly different notation), but omit their fairly straightforward proofs.

\begin{lemma}[Lemma 3.1 of \citet{prasadan2024informationtheoreticlimitsrobust}]\label{lemma:nonparam:pruned:tree:properties} For any $J\ge 3$, $\cL(J)$ is a $\tfrac{d}{2^{J-2}c}$-covering of $\cF$ and a $\tfrac{d}{2^{J-1}c}$-packing of $\cF$. In addition, for each $J\ge 2$ and any parent node $\Upsilon_{J-1}$ at level $J-1$, its offspring $\cO(\Upsilon_{J-1})$ form a $\tfrac{d}{2^{J-2}c}$-covering of the set $B(\Upsilon_{J-1}, \tfrac{d}{2^{J-2}}) \cap \cF$. Furthermore, the cardinality of $\cO(\Upsilon_{J-1})$ is upper bounded by $\cMFloc(\tfrac{d}{2^{J-2}}, 2c)$ for $J \geq 2$.
\end{lemma}

\begin{lemma}[Lemma 3.2 of \citet{prasadan2024informationtheoreticlimitsrobust}] \label{lemma:nonparam:bound:level:J:intersect:ball:mu} For any $f\in \cF$ and for any $J\ge 2$, the cardinality of $\cL(J-1)\cap B(f, \tfrac{d}{2^{J-2}})$ is upper bounded by $\cMFloc(\tfrac{d}{2^{J-2}}, c)\leq \cMFloc(\tfrac{d}{2^{J-2}}, 2c)$. 
\end{lemma}

\begin{lemma}[Lemma 3.3 of \citet{prasadan2024informationtheoreticlimitsrobust}] \label{lemma:nonparam:cauchy:sequence}  Let $[\Upsilon_1, \Upsilon_2,\dots,]$ be the nodes of an infinite path in $G$, i.e., $\Upsilon_{J + 1} \in \cO(\Upsilon_{J})$ for any $J \in \NN$ and $\Upsilon_1=f_R$. Then for any integers $J\ge J'\ge 1$, $\|\Upsilon_{J'}-\Upsilon_J\|_{\LtwoP}\le  \frac{d(2+4c)}{c 2^{J'}}$.
\end{lemma}

Algorithm \ref{algo_upperbound_nonparametric} then traverses the tree as follows. Starting with our root node $\Upsilon_1=f_R$, given an update $\Upsilon_{k}$ and its offspring set $\cO(\Upsilon_k)$, we set \[\Upsilon_{k+1}=\argmin_{f\in\cO(\Upsilon_k)} \sum_{i=1}^n (Y_i-f(X_i))^2.\] By Lemma \ref{lemma:nonparam:cauchy:sequence} and our assumption that $\cF$ is complete, we obtain a convergent Cauchy sequence. Observe that it is only at this stage that we actually use our data. 

This procedure may encounter ties in the minimization step, so we must modify our algorithm to handle such cases and, in doing so, preserve measurability of the estimator it outputs. Our goal is to induce a lexicographic ordering on our functions and then choose the smallest function in the ordering to break ties. The details are given in the appendix (Remark \ref{remark:breaking_ties}). We can also show (Lemma \ref{lemma:equivalent:norms}) that if an equivalent norm to $\|\cdot\|_{\LtwoP}$ is known, then the packing procedure can be performed in this norm instead and we will achieve the same necessary bounds.

Unfortunately, both the tree construction of Algorithm \ref{algorithm:nonparam:directed:tree} and the traversal in Algorithm \ref{algo_upperbound_nonparametric} are computationally intractable. The authors conjecture, however, that with a slight modification of our algorithm, the minimax rate can be achieved up to multiplicative logarithmic factors in polynomial time conditional on the dimension scaling logarithmically with $n$ and taking $\mathcal{F}$ to be linear functionals over projectable convex sets.

\begin{algorithm} 
\SetKwComment{Comment}{/* }{ */}
\caption{Upper Bound Algorithm for Nonparametric Regression \label{algo_upperbound_nonparametric}}
\KwInput{Root node $\Upsilon_1=f_R \in \cF$, pruned graph $G$, data $\{(X_i,Y_i)\}_{i=1}^n$}
$k \gets 1$\;
$\Upsilon \gets [\Upsilon_1]$\;
\While{TRUE} {
    $\Upsilon_{k+1} \gets \argmin_{f\in\cO(\Upsilon_k)} \sum_{i=1}^n (Y_i-f(X_i))^2$ \Comment*[r]{Break ties lexicographically}
    $k \gets k + 1$\;
}
\Return{$\Upsilon=[\Upsilon_1,\Upsilon_2,\Upsilon_3,\dots]$} 
\end{algorithm}

\subsection{Moment Condition}

Recall our assumption on both $\cF$ and $\PP_X$ given in  \eqref{new:bernstein:sufficient:condition}. Earlier work on the convex constrained non-parametric regression setting by \citet{mendelson_2017_local_versus_global} imposed an $L$-sub-Gaussian condition (defined in the following proposition). It turns out an $L$-sub-Gaussian assumption implies \eqref{new:bernstein:sufficient:condition}, as does having a bounded sup-norm. Similar to $L$-sub-Gaussianity, our new condition implies a Bernstein-like concentration inequality for $\|f-g\|_{L_2(\PP_n)}^2$ that we use to bound our algorithmic error.

We start by proving \eqref{new:bernstein:sufficient:condition} is more general in the following proposition and example. Note that the proof requires $\EE (f-g)^2 \leq d^2$, where $d$ is considered a fixed constant.

\begin{proposition} \label{proposition:L:subGaussian:implies:moment:condition} Assume the $\LtwoP$ diameter $d$ is some fixed constant. If $\cF$ and $\PP_X$ satisfy an $L$-sub-Gaussian assumption, i.e., \begin{equation}
    \|f - g\|_{L_p(\PP_X)} \leq L\sqrt{p} \|f - g\|_{\LtwoP} \label{eq:L:subgaussian:assumption}
\end{equation}for all $f,g\in\cF$ and $p\ge 2$ then \eqref{new:bernstein:sufficient:condition} holds. Moreover, if $\cF$ has a uniform sup-norm bound, we again have \eqref{new:bernstein:sufficient:condition}.
\end{proposition}

\begin{example}
    Suppose $X \sim \cN(0,1)$  and the model is $Y = \bar{f}(X) + \xi$, where \[\bar{f} \in \cF = \{f | f \mbox{ is 1-Lipschitz}, f(0) = 0\}.\] Then we have \eqref{new:bernstein:sufficient:condition} but this example is neither uniformly bounded in sup-norm nor $L$-sub-Gaussian.
\end{example}

The proof that the preceding example is not $L$-sub-Gaussian follows the logic in \citet{qiyang_wellner} using the failure of a necessary small-ball condition.

Next, we show our moment condition \eqref{new:bernstein:sufficient:condition} implies a Bernstein-like concentration inequality. Indeed, the proof is very similar to that of Bernstein's inequality. 

\begin{lemma} \label{lemma:moment:condition:implies:bernstein:concentration}
    Assume for a random variable $Z$ we have
\begin{align*}
    \EE Z^{2p} \leq \EE Z^2 L^{p-1} p!/2,
\end{align*}
for all integers $p\ge 2$ and for some $L > 0$. Then if $Z_1,\ldots, Z_n$ are i.i.d. variables with the same law as $Z$, we have
\begin{align*}
   \MoveEqLeft \PP\left(\sum_{i = 1}^n (Z_i^2 - \EE Z^2) \geq t\right) \vee \PP\left(\sum_{i = 1}^n (Z_i^2 - \EE Z^2) \leq -t\right) \\ &\leq \exp\left(-\frac{t^2}{16 n \EE Z^2 L + 4 Lt}\right).
\end{align*}
\end{lemma}

\subsection{Algorithm Error Bounds}

To prove our algorithm has the desired properties, we first relate the $\LtwoP$ norm to the $L_2(\PP_n)$ norm. In particular, Lemma \ref{lemma:bernstein:nonparam:v2} shows that if $f,g\in\cF$ are far apart but $f$ is close to $\bar{f}$ in $\LtwoP$-distances, then the same is true for $f,g,\bar{f}$ with high probability in the $L_2(\PP_n)$-distance. The proof relies on applying our Bernstein-like concentration inequality in Lemma \ref{lemma:moment:condition:implies:bernstein:concentration}. This result also applies when the noise is sub-Gaussian.

\begin{lemma}\label{lemma:bernstein:nonparam:v2} Suppose $f,g\in\cF$ satisfy $n\|f-g\|_{\LtwoP}^2\ge C^2\delta^2$ for some $C>0$. Suppose further that $n\|f-\bar{f}\|_{\LtwoP}^2<\delta^2$. Then $n\|f-\bar{f}\|_{L_2(\PP_n)}^2 \le 2\delta^2$ and  $n\|f-g\|_{L_2(\PP_n)}^2 \ge (C^2-1)\delta^2$ both simultaneously hold with probability at least \[1 -\exp\left(-\frac{\delta^2}{20L}\right)- \exp\left(-\frac{\delta^2}{4L(C^2 + 1)}\right). \]
\end{lemma}

The next lemma considers hypothesis testing whether the true function is $f$ or $g$ when $f,g\in\cF$ are at least $\LtwoP$-distance $\delta/\sqrt{n}$ apart. When $f$ is $\delta/\sqrt{n}$ close to the truth, the probability our test claims $g$ is the right choice is low, and similarly, when $g$ is close to the truth, the probability our test claims $f$ is the right choice is low. 

\begin{lemma}\label{lemma:testing:nonparam:gaussianbound} Suppose we are testing $H_0:\bar{f}=f$ vs $H_1:\bar{f}=g$ where $ n\|f-g\|_{\LtwoP}^2\ge C^2\delta^2$ for some $C>3$. Then the test $\psi(\vec{Y}) = \mathbbm{1}\big( \|\vec{Y}-f(\vec{X})\|_2^2 \geq \|\vec{Y}-g(\vec{X})\|_2^2\big)$  satisfies
\begin{align*}
    \MoveEqLeft\sup_{\bar{f}: n\|f-\bar{f}\|_{\LtwoP}^2<\delta^2}\PP_{\bar{f}}(\psi = 1) \vee \sup_{\bar{f}: n\|g-\bar{f}\|_{\LtwoP}^2<\delta^2}\PP_{\bar{f}}(\psi = 0) \\ &\le 3 \exp\left(-\Lambda(C,\sigma,L)\cdot\delta^2\right),
\end{align*} where \begin{align} \label{eq:nonparam:exp_triple_minimum}
    \Lambda(C,\sigma,L) = \min\left\{\frac{\left(\sqrt{C^2-1}-2\sqrt{2}\right)^2}{8\sigma^2}, \frac{1}{20L},\frac{1}{4L(C^2+1)} \right\}.
\end{align}
\end{lemma}

We now build on the previous lemma by showing the function in a $\delta$-covering closest to our data will also be $\delta$-close (in the $\LtwoP$-distance) to the true function with high probability. 

\begin{lemma} \label{lemma:nonparam:intermediate_error} Let $f_1,\dots,f_M$ be a $\delta$-covering of $\cF'\subseteq\cF$ and suppose $\bar{f}\in \cF'$. Let \[i^{\ast}= \argmin_{i\in[M]} \sum_{j=1}^n (Y_j - f_i(X_j))^2.\] Then  recalling the definition of $\Lambda(C,\sigma, L)$ in \eqref{eq:nonparam:exp_triple_minimum}, we have for $C>3$ that \begin{align*}\PP_{\bar{f}}\left(\|f_{i^{\ast}}- \bar{f}\|_{\LtwoP} > (C +1)\delta\right)  &\le 3 M\exp\left(-\Lambda(C,\sigma,L)\cdot n\delta^2\right).
\end{align*}
\end{lemma}

Note that the previous lemma still holds if we have an equivalent norm, i.e., $c_1 \|f-g\|_{\LtwoP} \leq \|f - g\|' \leq c_2 \|f-g\|_{\LtwoP}$ for some absolute constants $c_1,c_2>0$. As a consequence, we can build the local packing and covering sets with the norm $\|\cdot\|'$ in place of $\|\cdot\|_{\LtwoP}$ and obtain the same minimax rate. 

\begin{lemma} \label{lemma:equivalent:norms} Suppose there exists an equivalent norm $\|\cdot\|'$ on $\cF$ to $\|\cdot\|_{\LtwoP}$, i.e., there exists $c_1,c_2>0$ such that $c_1\|f-g\|_{\LtwoP}\le \|f-g\|' \le c_2\|f-g\|_{\LtwoP}$. Suppose $f_1,\dots, f_M$ is a $\delta$-covering in the $\|\cdot\|'$ norm of $\cF'\subseteq\cF$ with $\bar{f}\in\cF'$ and set $i^{\ast}= \argmin_{i\in[M]} \sum_{j=1}^n (Y_j - f_i(X_j))^2.$  Then so long as  $C>4c_1^{-1}c_2-1$, setting $\tilde C = c_1c_2^{-1}(C+1)-1$, we have \begin{align*}\PP_{\bar{f}}\left(\|f_{i^{\ast}}- \bar{f}\|' > (C +1)\delta\right)  &\le 3 M\exp\left(-c_1^{-2}\Lambda(\tilde C,\sigma,L)\cdot n\delta^2\right).
\end{align*} 
\end{lemma}
    \begin{proof}
        Observe that $f_1,\dots,f_M$ is a $c_1^{-1}\delta$-covering of $\cF'$ in the $\LtwoP$-norm. To see this, if $f\in \cF'$, there exists $f_i\in\cF'$ such that $\|f_i-f\|'\le \delta$. Hence \[\|f_i-f\|_{\LtwoP}\le c_1^{-1}\|f_i-f_j\|' \le  c_1^{-1}\delta.\]  Moreover, note that $\|f_{i^{\ast}}- \bar{f}\|' > (C +1)\delta$ implies $\|f_{i^{\ast}}- \bar{f}\|_{\LtwoP}> c_2^{-1}(C +1)\delta$. Thus, we have by Lemma \ref{lemma:nonparam:intermediate_error} that
        \begin{align*}
            \PP_{\bar{f}}\left(\|f_{i^{\ast}}- \bar{f}\|' > (C +1)\delta\right) &\le \PP_{\bar{f}}\left(\|f_{i^{\ast}}- \bar{f}\|_{\LtwoP}>c_2^{-1}(C +1)\delta \right) \\
            &=\PP_{\bar{f}}\left(\|f_{i^{\ast}}- \bar{f}\|_{\LtwoP}>c_1c_2^{-1}(C +1)c_1^{-1}\delta \right) \\
            &= \PP_{\bar{f}}\left(\|f_{i^{\ast}}- \bar{f}\|_{\LtwoP}>(\tilde C +1)c_1^{-1}\delta \right) \\
            &\le  3 M\exp\left(-c_1^{-2}\Lambda(\tilde C,\sigma,L)\cdot n\delta^2\right).
        \end{align*} We used $\tilde C = c_1c_2^{-1}(C+1)-1$ in the third line and then applied Lemma \ref{lemma:nonparam:intermediate_error} with $c_1^{-1}\delta$ and $\tilde C$. Note that $C>4c_1^{-1}c_2-1$ implies $\tilde C>3$.
    \end{proof}

\begin{remark} \label{subgaussian_remark_for_lemma_testing} Lemmas \ref{lemma:testing:nonparam:gaussianbound},  \ref{lemma:nonparam:intermediate_error}, and \ref{lemma:equivalent:norms} can be generalized to our sub-Gaussian setting where each $\xi_i\sim\mathrm{SG}(0,\sigma^2)$. The proof is in the appendix.
\end{remark}

Another property we will need is the monotonicity of the local metric entropy. It is easy to check this holds for convex sets (see, for instance, \cite[Lemma 9]{lecamshamindramatey2022} or \citet[Lemma II.8]{neykov2022minimax}). The proof in the general star-shaped setting is given in \citet[Lemma 1.4]{prasadan2024informationtheoreticlimitsrobust}, and easily transfers to our function class setting by just swapping out the Euclidean norm with our $\LtwoP$ norm.

\begin{lemma}\label{simple:lemma:monotone:nonparam}The function $\varepsilon \mapsto \cMFloc(\varepsilon, c)$ is monotone non-increasing.
\end{lemma}

Now we return to our algorithm and quantify its error in the squared $\LtwoP$-distance. The structure of the proof can be seen as a special case of the robust analogue in \citet[Theorem 3.8]{prasadan2024informationtheoreticlimitsrobust}, and an adaptation of \citet[Theorem II.10]{neykov2022minimax} (with some important corrections made to the argument in the latter). We then show in the subsequent theorem that the minimum of this error and the squared diameter is in fact the minimax rate. By Remark \ref{subgaussian_remark_for_lemma_testing}, the proof holds for sub-Gaussian noise as well.

\begin{theorem} \label{theorem:nonparam:upperbound:main}  Let $J^{\ast}$ be the maximal integer $J\ge 1$ with $\varepsilon_{J} = \frac{d\sqrt{\Lambda(c/2-1,\sigma,L)} }{2^{J-2}c}$ satisfying\begin{equation} \label{eq:upper_bound_condition:main_theorem} n\varepsilon_{J}^2 > 2\log \left[\cMFloc\left(\frac{\varepsilon_{J}\cdot c}{\sqrt{\Lambda(c/2 - 1,\sigma,L)}}, 2c\right)\right]^2\vee \log 2,\end{equation} and $J^{\ast}=1$ if this condition is never satisfied. Let $f^{\ast\ast}(\vec{X},\vec{Y})$ be the output of at least $J^{\ast}$ steps of  Algorithm \ref{algo_upperbound_nonparametric}. Then $\EE_{\vec{X},\vec{Y}}\|\bar{f} - f^{\ast\ast}(\vec{X},\vec{Y})\|_{\LtwoP}^2 \lesssim \varepsilon_{J^{\ast}}^2\wedge d^2$, up to absolute constants depending only on $c$, $L$, and $\sigma$. We defined $\Lambda(C,\sigma, L)$ in \eqref{eq:nonparam:exp_triple_minimum}, where $C=\tfrac{c}{2}-1$ and $c$ is the sufficiently large absolute constant we use in the definition of local metric entropy.
\end{theorem}

\begin{remark} \label{remark:swapping:2c:for:c} By similar logic to \citet[Remark 3.9]{prasadan2024informationtheoreticlimitsrobust}, in \eqref{eq:upper_bound_condition:main_theorem} we can without loss of generality replace $2c$ with $c$ in $\cMFloc$. This is because we could have used $2c$ in our lower bounds instead and only changed the result by absolute constants.
\end{remark}

\begin{theorem} \label{theorem:bounded:minimax:achieved} Define $\varepsilon^{\ast}:=\sup\{\varepsilon\ge 0:n\varepsilon^2 \le \log \cMFloc(\varepsilon, c)\}$ where $c$ is a sufficiently large absolute constant. Then the minimax rate is ${\varepsilon^{\ast}}^2\wedge d^2$ up to absolute constants.
\end{theorem}

\begin{remark}[Without Total Boundedness] \label{remark:total:boundedness} Consider the case where  the local packing number is infinite for some $\varepsilon$, i.e., there are arbitrarily large $\varepsilon/c$-packings of sets of the form $\cF\cap B(f,\varepsilon)$ for $f\in\cF$. Note that Lemma \ref{lemma:fano:nonparam} and Lemma \ref{lemma:nonparam:mutualinfo} still hold since they only require the existence of $\varepsilon$-separated sets. For the lower bound result in Lemma \ref{lemma:nonparam:lowerbound}, we adapt the proof by picking an $f$ such that the maximal $(\varepsilon/c)$-packing set of $\cF\cap B(f,\varepsilon)$ is infinite. Then we use a finite subset of that infinite packing set and apply the same logic as in the rest of the proof.

For the upper bound, the algorithm may require solving an optimization over infinitely many functions, but our results still hold. Lemma \ref{lemma:bernstein:nonparam:v2} and Lemma \ref{lemma:testing:nonparam:gaussianbound} do not rely on any packing constructions but simply a concentration result. Lemma \ref{lemma:nonparam:intermediate_error} is a statement about any packing set, not necessarily a maximal or infinite one. Monotonicity of the local metric entropy is unaffected. The arguments within the proof of Theorem \ref{theorem:nonparam:upperbound:main} and Theorem \ref{theorem:bounded:minimax:achieved} also hold either trivially (e.g., that the probability is bounded by infinity) or a key hypothesis such as \eqref{eq:upper_bound_condition:main_theorem} implies the local metric entropy is finite. The similar proofs in the subsequent adaptive setting will also be unaffected.
\end{remark}

\section{Adaptive Bound for Convex Classes}
\label{section:nonparam:adaptive}

We relax the definition of local metric entropy by removing the supremum and instead considering the pointwise value of the local metric entropies. This lets us produce an estimator that is adaptive to the true point $\bar{f}$, which is a difference from the work in \citet{yang1999information} and \citet{mendelson_2017_local_versus_global}, and to our knowledge is the first example of an adaptive estimator in this general setting. However, we no longer permit arbitrary star-shaped function classes and require convexity. In the subsequent section we will apply this adaptivity result to sparse vectors in the $\ell_1$ ball and monotone functions that are piece-wise constant on $k$ hyper-rectangles.

\begin{definition}[Adaptive Local Entropy]\label{local:adaptive:entropy:def:nonparam} Recall we defined $\cM(r,\cF)$ as the largest cardinality of a $r$-packing set in the $\LtwoP$ norm of $\cF$. Let $c>0$. Then we define \begin{align*}
    \cMFadloc(f, \varepsilon, c) = \cM(\varepsilon/c, B(f, \varepsilon) \cap \cF).
\end{align*} Then $\log \cMFadloc(f, \varepsilon, c)$ is the adaptive local entropy of $\cF$ at $f$. 
\end{definition}

The adaptive local entropy appears in several existing works, including \citet{agglomeration_yang_2004} and \citet{wang_yang_2014}. A following lemma relates the adaptive local metric entropy of two nearby points.

\begin{lemma} \label{lemma:adloc:technical} Let $f,g\in\cF$ satisfy $\|f-g\|_{\LtwoP}\le \delta$. Then $\cMFadloc(f, \varepsilon, c) \le  \cMFadloc(g, 2\varepsilon, 2c)$ provided $\varepsilon\ge\delta$.
\end{lemma}
    \begin{proof}
        If we show $B(f,\varepsilon)\subseteq B(g,2\varepsilon)$, then the cardinality of the largest $(\varepsilon/c)$-packing of $B(f,\varepsilon)\cap\cF$ will be less than or equal to the cardinality of the largest $(2\varepsilon/2c)$-packing of $B(g,2\varepsilon)\cap\cF$. Well, if $h\in B(f,\varepsilon)$, then $\|h-f\|_{\LtwoP}\le \varepsilon$. Using the triangle inequality, \[\|h-g\|_{\LtwoP} \le \|h-f\|_{\LtwoP} +\|f-g\|_{\LtwoP} \le \varepsilon + \delta  \le 2\varepsilon,\] showing that $h\in B(g,2\varepsilon)$.
    \end{proof}

\begin{lemma}\label{lemma:monotone:nonparam:adaptive}Assume $\cF$ is convex and pick $f\in\cF$. Then the function $\varepsilon \mapsto \cMFadloc(f,\varepsilon, c)$ is non-increasing, and the function $c \mapsto \cMFadloc(f,\varepsilon, c)$ is non-decreasing. 
\end{lemma} 
    \begin{proof} We prove the first claim as the second has a similar argument. Pick $0<\varepsilon_1 < \varepsilon_2$. Let $f_1,\dots,f_N$ be a maximal $\varepsilon_2/c$ packing set of $B(f,\varepsilon_2)\cap\cF$, so that $N=\cMFadloc(f,\varepsilon_2,c)$. For each $1\le i\le N$, set $g_i=\alpha f_i + (1-\alpha)f$ where $\alpha = \tfrac{\varepsilon_1}{\varepsilon_2}\in(0,1)$, noting $g_i\in\cF$ by convexity. Then for each $i$ we have \begin{align*}
            \|g_i-f\|_{\LtwoP} &= \left\|\tfrac{\varepsilon_1}{\varepsilon_2}f_i+(1-\tfrac{\varepsilon_1}{\varepsilon_2})f-f\right\|_{\LtwoP} \\ &= \tfrac{\varepsilon_1}{\varepsilon_2}\|f_i-f\|_{\LtwoP} \\ &\le \tfrac{\varepsilon_1}{\varepsilon_2}\cdot\varepsilon_2=\varepsilon_1,
        \end{align*} using that $f_i\in B(f,\varepsilon_2)$. Thus, $g_1,\dots,g_N\in B(f, \varepsilon_1)\cap\cF$. Moreover, for $i\ne j$, we have \begin{align*}
            \|g_i-g_j\|_{\LtwoP} = \alpha\|f_i-f_j\|_{\LtwoP}>\tfrac{\varepsilon_1}{\varepsilon_2}\cdot \tfrac{\varepsilon_2}{c}=\tfrac{\varepsilon_1}{c},
        \end{align*} verifying $g_1,\dots,g_N$ form a (possibly non-maximal) $\varepsilon_1/c$-packing of $B(f, \varepsilon_1)\cap\cF$. Hence \[\cMFadloc(f,\varepsilon_1, c) \ge N=\cMFadloc(f,\varepsilon_2, c).\]
    \end{proof}

We now show that the main estimator is adaptive to the true point $\bar{f}$, in which the  key condition \eqref{eq:upper_bound_condition:main_theorem:adaptive} does not require taking a supremum over all $f\in\cF$ in the metric entropy term.  The $J^{\dagger}$ and $J^{\ast}$ that are defined below therefore differ from the $J^{\ast}$ of Theorem \ref{theorem:nonparam:upperbound:main}. As before, our proof also applies in the sub-Gaussian noise setting.

\begin{theorem} \label{theorem:nonparam:upperbound:main:adaptive}   Assume $\cF$ is convex. Let $J^{\dagger}$ be the maximal integer $J\ge 1$ with $\varepsilon_{J} = \frac{d\sqrt{\Lambda(\tfrac{c}{2}-1, \sigma, L)} }{2^{J-2}c}$ satisfying \begin{equation} \label{eq:upper_bound_condition:main_theorem:adaptive:preliminary} n\varepsilon_{J}^2 > \inf_{f\in\cF} 2\log \left[\cMFadloc(f, \tfrac{c\varepsilon_J}{\sqrt{\Lambda(\tfrac{c}{2}-1, \sigma, L)}},4c)\right]^2\vee \log 2\end{equation} and $J^{\dagger}=1$ if this condition is never satisfied. Let $f^{\ddagger}(\vec{X},\vec{Y})$ be the output of $J^{\ddagger}$ steps of
  Algorithm \ref{algo_upperbound_nonparametric} where $J^{\ddagger}\ge J^{\dagger}$.

Next, let $J^{\ast}$ be the maximal integer $J\ge 1$ with $\varepsilon_{J}$ satisfying \begin{equation} \label{eq:upper_bound_condition:main_theorem:adaptive} n\varepsilon_{J}^2 >  2\log \left[\cMFadloc(\bar{f}, \tfrac{c\varepsilon_J}{\sqrt{\Lambda(\tfrac{c}{2}-1, \sigma, L)}},4c)\right]^2\vee \log 2,\end{equation} noting the use of the true $\bar{f}\in\cF$. Set $J^{\ast}=1$ if this condition never occurs. Then $\EE_{\vec{X},\vec{Y}}\|\bar{f} - f^{\ddagger}(\vec{X},\vec{Y})\|_{\LtwoP}^2 \lesssim \varepsilon_{J^{\ast}}^2$ up to some universal constants, $L$, and $\sigma$. We defined $\Lambda(C, \sigma, L)$ in \eqref{eq:nonparam:exp_triple_minimum}, where $C=\tfrac{c}{2}-1$ and $c$ is some sufficiently large universal constant.
\end{theorem}

\section{Examples}
\label{section:nonparam:examples}

We now give some examples using our theory, most of which are known. The known examples include linear functionals over the $\ell_1$ ball, Lipschitz classes, and monotone functions in a multivariate setting. However, we obtain new minimax rates for linear functionals over ellipsoids. 

\subsection{Linear Functionals} \label{subsection:linear_functionals}

We begin with function spaces  defined as linear functionals over compact, star-shaped sets $K$ using sub-Gaussian isotropic data; such function spaces will not have a uniform bound on the sup-norm. For such spaces, local metric entropies in our function space using an $\LtwoP$-norm will be equivalent to local metric entropies in $\RR^p$ with the Euclidean 2-norm. We then specialize our results to the $\ell_1$-ball and ellipsoids, whose convexity enables use of our adaptive results.

We call a random variable $Z\in\RR^p$ isotropic if $\EE[Z]=0$ and $\EE[ZZ^T]=\mathbb{I}_p$. Moreover, we say $Z\in\RR^p$ is a sub-Gaussian random vector with mean $\mu\in\RR^p$ and variance proxy $\tau^2$ if for any $v\in\RR^p$ with $\|v\|_2=1$, the random variable $v^T(Z-\mu)$ is sub-Gaussian with mean $0$ and variance proxy $\tau^2$. 
 
Let $K\subset\RR^p$ be a compact, star-shaped set. Define $\cF=\{f_{\beta}\colon \RR^p\to \RR, f_{\beta}(x)=x^T\beta;\beta\in K\}$. Now let us draw i.i.d. isotropic sub-Gaussian vectors $X_1,\dots,X_n$ with mean $0$ and variance proxy $\tau^2>1$. Importantly, functions in $\cF$ are not bounded in sup-norm despite $K$ being compact. Observe that \begin{align*}
   \diam(\cF)^2 &= \sup_{f_{\beta_1},f_{\beta_2}\in\cF} \|f_{\beta_1}-f_{\beta_2}\|_{\LtwoP}^2 =\sup_{\beta_1,\beta_2\in K} \EE\left[(\beta_1-\beta_2)^TX\right]^2 \\
   &= \sup_{\beta_1,\beta_2\in K} \|\beta_1-\beta_2\|_2^2 = \diam(K)^2.
\end{align*} The third equality used the fact that $X$ is isotropic. Thus, the diameter of $\cF$ in the $\LtwoP$-norm is the diameter of $K$ which is finite.

Next, note that for any distinct $\beta_1,\beta_2\in K$, we have $\|\beta_1-\beta_2\|_2^{-1}\cdot (\beta_1-\beta_2)^TX\sim \mathrm{SG}(0,\tau^2)$. Let $W$ denote this quantity. By \citet[Proposition 2.5.2(ii)]{vershynin2018high}, for some absolute constant $\alpha>0$ we have $\EE|W^p| \le \tau^p \alpha^p p^{p/2}$. Then we compute that \begin{align*}
  \|f_{\beta_1}-f_{\beta_2}\|_{L_p(\PP_X)}^p &= \EE\left[(\beta_1-\beta_2)^TX\right]^p  \le  \|\beta_1-\beta_2\|_2^{p}\cdot \EE|W^p| \\ &\le  \|\beta_1-\beta_2\|_2^{p} \cdot \tau^p \alpha^p p^{p/2}.
\end{align*} Taking the $p$th root and noting by our above argument that $X$ being isotropic implies $\|\beta_1-\beta_2\|_2 = \|f_{\beta_1}-f_{\beta_2}\|_{\LtwoP}$, we have \[\|f_{\beta_1}-f_{\beta_2}\|_{L_p(\PP_X)} \le \tau\alpha  \|f_{\beta_1}-f_{\beta_2}\|_{\LtwoP}\sqrt{p}.\] Hence the $L$-sub-Gaussian condition in \eqref{eq:L:subgaussian:assumption} holds, so by Proposition \ref{proposition:L:subGaussian:implies:moment:condition}, our key condition \eqref{new:bernstein:sufficient:condition} holds.

Moreover, since $ \|f_{\beta_1}-f_{\beta_2}\|_{\LtwoP} = \|\beta_1-\beta_2\|_2$ holds for all $\beta_1,\beta_2\in K$ and since $B(f_{\beta},\varepsilon,\|\cdot\|_{\LtwoP})= \{f_{\beta'}:\beta'\in B_2^p(\beta,\varepsilon)\}$, it follows that any $\varepsilon/c$-packing of $B(f_{\beta},\varepsilon,\|\cdot\|_{\LtwoP})\cap\cF$ corresponds to a $\varepsilon/c$-packing of $B_2^p(\beta,\varepsilon)\cap K$ in the Euclidean 2-norm on $\RR^p$. Hence $\cMFloc(\varepsilon,c) = \cM_{K}^{\loc}(\varepsilon,c)$.

\subsubsection{Linear Functionals on the \texorpdfstring{$\ell_1$}{}-ball}

Let $K=B_1^p(0,1):= \{x\in\RR^p:\|x\|_1\le 1\}$ be the $\ell_1$-unit-ball. Let $\SS_1^{p-1} = \{x\in\RR^p:\|x\|_1= 1\}$ be the $\ell_1$-unit sphere, and similarly set $\SS_2^{p-1} = \{x\in\RR^p:\|x\|_2= 1\}$. We again take $\cF$ to be the set of linear functionals over $K$, noting that both $K$ and $\cF$ are  convex. Note that the equivalence of local metric entropies between our function space and $K$ in our linear functionals setting with isotropic and sub-Gaussian data also applies to the adaptive local metric entropy. Suppose $\beta^{\ast}$ is $s$-sparse, i.e., has at most $s$ non-zero components, and that $\beta^{\ast}\in \SS_1^{p-1}$. Our goal will be to use the adaptive local entropy to derive a minimax rate upper bound for $\cF$ under this sparsity assumption. Our result will match that of \citet{raskutti2011minimax}.

We will need some tools from \citet{chandrasekaran2012convex}. The authors begin with an atomic set $\cA$ which is a compact subset of $\RR^p$. The authors then construct an atomic norm induced by this atomic set using the gauge function, assuming $\cA$ is centrally symmetric about the origin. For our example, we take $\cA$ to be the set of $1$-sparse vectors of unit norm, so that the convex hull $\mathrm{conv}(\cA)$ is precisely the $\ell_1$-unit-ball $B_1^p(0,1)$. We then define the atomic norm induced by $\cA$ at a vector $x$ by $\|x\|_{\cA}=\inf\{t>0:x\in t\cdot \mathrm{conv}(\cA)\}$. Then clearly $\|x\|_{\cA}=\|x\|_1$ for $x\ne 0$, since the smallest $t>0$ such that $x/t\in B_1^p(0,1)$ is $\|x\|_1$, and $\|0\|_{\cA}=\|0\|_1$ since $0\in t\cdot B_1^p(0,1)$ for arbitrarily small $t>0$. Thus, our atomic set and atomic norm correspond exactly to the $\ell_1$-ball and $\ell_1$-norm. 

Next, recalling that the cone of a set $S\subseteq \RR^p$ is defined by $\mathrm{cone}(S):= S:=\{\alpha \cdot s:s\in S, \alpha >0\}$, we define the tangent cone at $x\in\RR^p$ as \[T_{\cA}(x)=\mathrm{cone}\{z-x:\|z\|_{\cA}\le \|x\|_{\cA}\} = \mathrm{cone}\{z-x:\|z\|_1\le \|x\|_1\}.\]

Let us compare the set $S_1:=\frac{B_2^p(\beta^{\ast},\varepsilon)\cap B_1^p(0,1)-\beta^{\ast}}{\varepsilon}$ to $S_2:=T_{\cA}(\beta^{\ast})\cap B_2^p(0,1)$, and show $S_1\subseteq S_2$. Suppose $x\in S_1$. Then $\varepsilon x + \beta^{\ast}\in B_2^p(\beta^{\ast},\varepsilon)\cap B_1^p(0,1)$. This implies that $\|\varepsilon x +\beta^{\ast}\|_1\le 1=\|\beta^{\ast}\|_1$ and $\|\varepsilon x\|_2\le \varepsilon$. The latter implies $\|x\|_2\le 1$, i.e., $x\in B_2^p(0,1)$. Next, observe that \[x = \frac{1}{\varepsilon}\left(\varepsilon x +\beta^{\ast} - \beta^{\ast}\right)\in \mathrm{cone}\{z-\beta^{\ast}:\|z\|_1\le \|\beta^{\ast}\|_1\}=T_{\cA}(\beta^{\ast}),\] where we took $z = \varepsilon x +\beta^{\ast}$. Thus, we conclude $x\in S_2$, so $S_1\subseteq S_2$.

For any set $T$, define its Gaussian width $w(T)$ by \[w(T) := \EE_{g\sim \cN(0,\mathbb{I}_p)}\Big[\sup_{t \in T}\langle t,g\rangle \Big].\] We will use numerous properties of the Gaussian width. Namely, $w$ is invariant under translations, is positive homogeneous, i.e., $w(\alpha T)=\alpha w(T)$ for any $\alpha>0$, and if $T_1\subseteq T_2$, then $w(T_1)\le w(T_2)$. Since the family of random variables $\{\langle t,g\rangle:t\in T\}$ is an example of a zero-mean Gaussian process, we apply the Sudakov minoration lower bound \citep[Theorem 5.30]{wainwright2019high} to upper bound the local metric entropy. 
\begin{align*}
    \log \mathcal{M}_{K}^{\adloc}(\beta^{\ast},\varepsilon,c) &=
    \log \cM(\varepsilon/c, B_2^p(\beta^{\ast},\varepsilon)\cap B_1^p(0,1)  )  \\
    &\le  \frac{4c^2}{\varepsilon^2} \cdot  \left[w(B_2^p(\beta^{\ast},\varepsilon)\cap B_1^p(0,1))\right]^2 \\
    &= 4c^2\cdot \left[w\left( \frac{B_2^p(\beta^{\ast},\varepsilon)\cap B_1^p(0,1) - \beta^{\ast}}{\varepsilon}  \right)\right]^2\\ 
    &\le 4c^2 \left[w(T_{\cA}(\beta^{\ast})\cap B_2^p(0,1))\right]^2.
\end{align*}  
We now relate the Gaussian width of  $T_{\cA}(\beta^{\ast})\cap B_2^p(0,1)$ to $T_{\cA}(\beta^{\ast})\cap \SS_2^p(0,1)$, and do so by using the statistical dimension of a convex cone $C$, as is defined in \citet[Proposition 3.1, (5)]{amelunxen2014living}: \begin{equation} \label{eq:statistical:dimension}
    \delta(C) = \EE_{g\sim \cN(0,\mathbb{I}_p)}\left[\left(\sup_{t \in C\cap B_2^p(0,1)}\langle t,g\rangle\right)^2 \right].
\end{equation} Applying \citet[Proposition 10.2]{amelunxen2014living}, 
    \begin{align*}
       \delta(T_{\cA}(\beta^{\ast}))&\le  [w(T_{\cA}(\beta^{\ast})\cap \SS_2^{p-1}(0,1))]^2 + 1.
    \end{align*} 
Additionally by \citet[Proposition 3.10]{chandrasekaran2012convex},  \begin{align*}
    [w(T_{\cA}(\beta^{\ast})\cap \SS_2^{p-1}(0,1))]^2 &\le 2s\log\left(\frac{p}{s}\right)+\frac{5}{4}s,
\end{align*} where we recall $\beta^{\ast}$ is $s$-sparse. Using Jensen's inequality and the fact that $T_{\cA}(\beta^{\ast})$ is a convex cone along with the previous two results,  we obtain 
    \begin{align*}
    \log \mathcal{M}_{K}^{\operatorname{adloc}}(\beta^{\ast},\varepsilon,c) &\le 4c^2\left[w(T_{\cA}(\beta^{\ast})\cap B_2^p(0,1))\right]^2 \\ &= 4c^2\left(\EE_{g\sim \cN(0,\mathbb{I}_p)}\left[\sup_{t \in T_{\cA}(\beta^{\ast})\cap B_2^p(0,1)}\langle t,g\rangle \right]\right)^2 \\
    &\le 4c^2\EE_{g\sim \cN(0,\mathbb{I}_p)}\left[\left(\sup_{t \in T_{\cA}(\beta^{\ast})\cap B_2^p(0,1)}\langle t,g\rangle\right)^2 \right] \\
    &= 4c^2\delta(T_{\cA}(\beta^{\ast})). \\ 
    &\le 4c^2[w(T_{\cA}(\beta^{\ast})\cap \SS_2^{p-1}(0,1))]^2 + 4c^2 \\
    &\le  4c^2s\left( 2\log\left(\frac{p}{s}\right)+\frac{5}{4}\right)+2c^2 \\ &\lesssim s\log(p/s).
    \end{align*} 
Hence,  $\log \cMFadloc(f_{\beta^{\ast}},\varepsilon,c)\lesssim s\log(p/s).$ The result holds if we replace $c$ with $4c$ as well.

Now recalling the definition of $\varepsilon_J$ and $J^{\ast}$ from  Theorem \ref{theorem:nonparam:upperbound:main:adaptive}, we will have for some constant $\kappa>0$ that \begin{align*}
    n(\varepsilon_{J^{\ast}}/2)^2 &\le 2\log \cMFadloc(f_{\beta^{\ast}},\kappa\cdot\varepsilon_{J^{\ast}},4c)\vee 
\log 2 \lesssim s\log(p/s).
\end{align*} Then using Theorem \ref{theorem:nonparam:upperbound:main:adaptive} and rearranging the previous inequality, we conclude 
\[\EE_{\vec{X},\vec{Y}}\|f_{\beta^{\ast}}- f^{\ddagger}(\vec{X},\vec{Y})\|_{\LtwoP}^2 \lesssim \varepsilon_{J^{\ast}}^2 \lesssim\frac{s\log(p/s)}{n},\] where $f^{\ddagger}(\vec{X},\vec{Y})$ is the output of our algorithm. This result matches that of \citet[Theorem 2(b)]{raskutti2011minimax}, although notably our estimator is different, our bound is in expectation rather than in high probability and we use data that are isotropic and sub-Gaussian.

\subsubsection{Linear Functionals over Ellipsoids}

The following example yields novel minimax rates for linear functionals over a class of ellipsoids. Define the matrix $K_c$ to be the $p\times p$ diagonal matrix with $i$th entry $a_i^{-1}$, where $0< a_1\le a_2\le\dots\le a_p$, where $a_p \in \RR$ is assumed to be a bounded constant independent of the sample size. Let $\Theta = \{\theta\in\RR^p:\theta^TK_c\theta\le 1\}$. Our assumption on $a_p$ implies that the set $\Theta$ has a bounded diameter. We draw $X_1,\dots,X_n$ isotropic and sub-Gaussian with variance proxy $\tau^2>1$, and for $1\le i\le n$ we observe $Y_i = X_i^T\theta^{\ast}+w_i$ where $w_i\sim \cN(0,\sigma^2)$ and $\theta^{\ast}\in\Theta$ is the true value of the parameter. Let $X\in\RR^{n\times p}$ be the matrix with $X_i$ in the $i$th-row. Our function class is therefore $\cF=\{f_{\theta}(x) = x^T\theta:\theta\in\Theta\}$. Note that $\diam(\cF)= \diam(\Theta)=2\sqrt{a_p}$. As we showed in Section \ref{subsection:linear_functionals}, since the $X_i$ are isotropic, $\|f_{\theta_1} - f_{\theta_2}\|_{\LtwoP}^2 = \|\theta_1-\theta_2\|_2^2$. 

Let us now compute the minimax rate for estimating $\theta^{\ast}$, or equivalently, $f_{\theta^{\ast}}$. We will follow \citet{neykov2022minimax} which has a very similar proof. Consider the Kolmogorov width $W_k(\Theta) = \min_{P\in\cP_k}\max_{\theta\in \Theta}\|P\theta-\theta\|,$ for $\cP_k$ the set of linear projections into $k$-dimensional subspaces. One can check that $W_k(\Theta) = \sqrt{a_{p-k}}$ where we take $a_0=0$. 

If $a_p>1/n$, then let $k\in[p]$ be such that  $a_{p-k}\le (k+1)/n$ and $a_{p-k+1}>k/n$. This will imply $W_k(\Theta)^2 \le (k+1)/n$. We claim that $\varepsilon^{\ast} \asymp\sqrt{k/n}$, where $\varepsilon^{\ast}$ is as defined in \eqref{definition:epsilon:star}. As a result, the minimax rate will be $k/n\wedge a_p$ if $a_p>1/n$, and $a_p$ if $a_p\le 1/n$, all up to constants.

We start with the upper bound. In the case $a_p \le 1/n$, note that the minimax rate is trivially bounded by the squared diameter $4a_p$. Suppose instead $a_p >1/n.$ Now take $\varepsilon^2\ge Ck/n$.  Pick any $f_{\theta}\in\cF$ and let $M_{f_{\theta}}$ be an $\varepsilon/c$-packing of $\cF\cap B(f_{\theta}, \varepsilon)$. Then this corresponds to some $\varepsilon/c$-packing $M_{\theta}$ of the set $\Theta\cap B_2^p(\theta,\varepsilon)$, where we are using the $p$-dimensional Euclidean $\ell_2$-norm in the ball centered at $\theta$. Let $P\in\cP_k$ solve the outer minimization problem in the Kolmogorov width definition for $\theta$.

Then for any $g_{\theta_1}, g_{\theta_2}\in M_{f_{\theta}}$, we have \begin{align*}
    \frac{\varepsilon}{c} &\le \|g_{\theta_1}- g_{\theta_2}\|_{\LtwoP} = \|\theta_1-\theta_2\|_2 \\ &\le \|\theta_1-P\theta_1 -\theta_2+P\theta_2\|_2+\|P\theta_1-P\theta_2\|_2 \\
    &\le 2W_k(\Theta)+\|P\theta_1-P\theta_2\|_2 .
\end{align*} Note that $\varepsilon^2 \ge Ck/n >\frac{C(k+1)}{2n} \ge \frac{C}{2}\cdot W_k(\Theta)^2,$ which means $2W_k(\Theta)\le \sqrt{\frac{8}{C}}\cdot \varepsilon$ so by taking $C$ sufficiently large, we have \begin{align*}
    \|P\theta_1-P\theta_2\|_2 &\ge \frac{\varepsilon}{c}  - 2W_k(\Theta) > \frac{\varepsilon}{c}- \sqrt{\frac{8}{C}}\cdot \varepsilon = \frac{\varepsilon}{C'}
\end{align*} for some constant $C'$. This means we have formed an $\frac{\varepsilon}{C'}$-packing set $\{P\theta:\theta\in M_{\theta}\}$ of the $k$-dimensional set $P\Theta$. On the other hand, a $k$-dimensional set contained in a $k$-dimensional sphere has a packing set with log cardinality bounded by $kC''$ for some constant $C''$. Since $\|P\theta_1-P\theta_2\|_2 >0$ for each $\theta_1,\theta_2\in M_{\theta}$, we have $|M_{\theta}| =|\{P\theta:\theta\in M_{\theta}\}|$. Hence the packing set $M_{\theta}$ has log cardinality bounded by $kC''$, as does $M_{f_\theta}$ by the equivalence of the Euclidean norm and $\LtwoP$ norm.

Hence we have $n\varepsilon^2 \ge kC$ while $\log \cMFloc(\varepsilon, c) \le kC''$, so taking $C$ large enough, we have $n\varepsilon^2 > \log \cMFloc(\varepsilon, c).$ But since $n\varepsilon^2$ is increasing in $\varepsilon$ while $\log \cMFloc(\varepsilon, c)$ is non-increasing in $\varepsilon$, we must have $\varepsilon^{\ast}\le\varepsilon$. Since this holds for all $\varepsilon\ge \sqrt{Ck/n}$, it follows that $\varepsilon^{\ast}\le \sqrt{Ck/n}$. The minimax rate will also be bounded by the squared diameter, so the minimax rate is therefore upper-bounded up to constants by $(k/n)\wedge 4a_p$ when $a_p>1/n$.

Now we prove the lower bound. We first handle the $a_p>1/n$ case. Consider the subset $\Theta'=\{\theta\in\RR^p:\sum_{i=p-k+1} \theta_i^2/a_i\le 1, \theta_1=\dots=\theta_{p-k}=0\}\subseteq\Theta$, i.e., we set the first $p-k$ coordinates of $\theta\in\Theta$ to $0$. But as $a_{p-k+1}>k/n$ and $a_1\le a_2\le\dots\le a_p$, we conclude that $\Theta'$ and therefore $\Theta$ contains a $k$-dimensional ball of radius $\sqrt{k/n}$. Hence $\log \cMFloc(\varepsilon, c) = \log \cM_{\Theta}^{\loc}(\varepsilon, c)\ge Ck$ for some constant $C$. Now take $\varepsilon = C'\sqrt{k/n}$ for some sufficiently small constant $C'$ such that $n\varepsilon^2 = (C')^2k \le Ck \le \log \cMFloc(\varepsilon, c)$. Since $\varepsilon^{\ast}$ is the supremum of all such $\varepsilon$ satisfying $n\varepsilon^2 \le \log \cMFloc(\varepsilon, c)$, we conclude $\varepsilon^{\ast}\ge \varepsilon \gtrsim\sqrt{k/n}$. Combined with the upper bound, we conclude in the case where $a_p>1/n$ that ${\varepsilon^{\ast}}^2\asymp k/n$ and thus the minimax rate is $(k/n)\wedge 4 a_p$ up to constants.

Lastly, we consider the lower bound for when $a_p\le n^{-1}$, which implies $a_i\le n^{-1}$ for all $i\in[p]$. Now take $\varepsilon = \frac{2\sqrt{a_p}}{C\sqrt{n}}$ for some constant $C$. For sufficiently large $C$, we have $\log \cMFloc(\varepsilon, c)>1 > \frac{4 a_p}{C^2}=n\cdot \varepsilon^2.$ Again, by the definition of $\varepsilon^{\ast}$ as a supremum, we have $\varepsilon^{\ast}\ge \varepsilon= \frac{2\sqrt{a_p}}{C\sqrt{n}}$, so $\varepsilon^{\ast}\gtrsim \sqrt{a_p}$, noting that $\sqrt{a_p}$ is the diameter up to constants. So the minimax rate is lower bounded by $a_p$ up to constants, and we already know it is upper bounded by it in this case. This completes our proof that the minimax rate is $(k/n)\wedge 4a_p$ if $a_p>1/n$, and $a_p$ if $a_p\le 1/n$, where $k\in[p]$ is such that  $a_{p-k}\le (k+1)/n$ and $a_{p-k+1}>k/n$.

Now, let us compare our result to that of \citet{pathak2023noisy}. Define \[\gamma_n = \sup_{\Omega\succ 0}\left\{\EE\;\mathrm{Tr}\left(\Omega^{-1}+ n^{-1}X^TX\right)^{-1}: \mathrm{Tr}(K_c^{-1/2}\Omega K_c^{1/2})\le n \right\},\] where we recall $K_c$ is the diagonal matrix with $(i,i)$-entry $a_i$ for each $1\le i \le p$. The authors then conclude the minimax rate is of order $\gamma_n/n$. 

Thus, if $a_p\le n^{-1}$, we equate the two results and conclude that $a_p\asymp \gamma_n/n$, or if $a_p>n^{-1}$ we conclude $(k/n)\wedge a_p \asymp \gamma_n/n$. Although calculating $\gamma_n$ from its definition may be challenging, our approach with local metric entropies is able to indirectly recover it.

\subsection{Lipschitz Classes}

The following example of Lipschitz classes recovers the known rate  given in \citet{stone_1982}. Given any function $f\colon[0,1]\to\RR$, we extend $f$ to $(1,\infty)$ by setting $f(x)=f(1)$ for $x>1$. For any such function $f$ and real number $h>0$, define the function $\Phi_{f,h}\colon[0,1]\to\RR$ by $\Phi_{f,h}(x)=f(x+h)-f(x)$. Now for each $\alpha\in(0,1]$ and $\gamma>0$, define \[\cF_{\alpha, \gamma}= \{f\colon [0,1]\to \RR : \|f\|_{\LtwoP}\le \gamma, \|\Phi_{f,h}\|_{\LtwoP} \le \gamma h^{\alpha} \text{ for all }h>0\}.\] Then this is a convex function class with bounded diameter, i.e., $\gamma \le \diam(\cF_{\alpha, \gamma})\le 2\gamma$. 

Then we have $\log \cM(\varepsilon, \cF_{\alpha,\gamma})\asymp \varepsilon^{-1/\alpha}$ by \citet[Section 6.4]{yang1999information} and using the inequality \begin{align*}
    \log \cM(\varepsilon/c, \cF_{\alpha, \gamma}) \ge   \log \cM_{\cF_{\alpha, \gamma}}^{\loc}(\varepsilon, c) \ge \log \cM(\varepsilon/c, \cF_{\alpha, \gamma}) -\log \cM(\varepsilon,\cF_{\alpha,  \gamma}),
\end{align*} we conclude $ \log \mathcal{M}_{\cF_{\alpha, \beta, \gamma}}^{\loc}(\varepsilon, c)\asymp \varepsilon^{-1/\alpha}$. The relationship $n\varepsilon^2\asymp \varepsilon^{-1/\alpha}$ occurs when $\varepsilon \asymp n^{-\frac{\alpha}{2\alpha + 1}}$, so we conclude the minimax-rate is determined by $n^{-\frac{\alpha}{2\alpha + 1}} \wedge \diam(\cF_{\alpha,  \gamma}) \asymp n^{-\frac{\alpha}{2\alpha + 1}} \wedge \gamma.$ 

\subsection{Monotone functions in general dimensions}

We now use results from \citet{gao_Wellner_monotone} to derive the known minimax rate for monotone functions in a multivariate setting. We let $\cF_p=\{f\colon [0,1]^p \to[0,1], f \text{ non-decreasing in each variable}\}$, where we say $f\colon [0,1]^p \to[0,1]$ is non-decreasing if \[f(x_1,\dots,x_{i-1},x_i,x_{i+1},\dots, x_p) \le f(x_1,\dots,x_{i-1}, x_i+y_i,x_{i+1},\dots, x_p)\] for any $(x_1,\dots,x_p)\in [0,1]^p$, $y_i> 0$, and $1\le i \le p$. Using the least squares estimator, \citet{isotonic_general_dimensions} achieve an error rate of $n^{-1/p}\log^4 n$ for $p\ge 3$, which is the minimax rate up to logarithm factors. In fact, under some mild conditions \citet{deng_and_zhang_2020} achieve $n^{-1/p}$ without the logarithmic factors using a block estimator. We now explain how our algorithm achieves the same rate. We also demonstrate adaptivity of our algorithm for functions that are piecewise constant on $k$ hyper-rectangles as is done in \citet{isotonic_general_dimensions} and \citet{deng_and_zhang_2020}.

We restate a key theorem from \citet{gao_Wellner_monotone}:

\begin{lemma}[Theorem 1.1 of \cite{gao_Wellner_monotone}]  For $p>2$, we have \[\log \cM(\varepsilon, \cF_p) \asymp \varepsilon^{-2(p-1)}.\] If $p=2$, then  \[\varepsilon^{-2}\lesssim\log \cM(\varepsilon, \cF_p) \lesssim \varepsilon^{-2}(\log 1/\varepsilon)^2.\]
\end{lemma} 

Assume $p>2$ at first. Recall once more by \citet[Section 6.4]{yang1999information} that \[\log \cM(\varepsilon/c,\cF_p)\ge \log \cM_{\cF_p}^{\loc}(\varepsilon,c)\ge \log \cM(\varepsilon/c,\cF_p) - \log \cM(\varepsilon,\cF_p).\] Therefore $\log \cM_{\cF_p}^{\loc}(\varepsilon,c)\asymp \varepsilon^{-2(p-1)}$. Next, the relationship $n\varepsilon^2 \asymp \varepsilon^{-2(p-1)}$ holds when $\varepsilon \asymp n^{-1/(2p)}$. Since the diameter of $\cF_p$ is 1, we conclude the minimax rate for $p>2$ is given by $n^{-1/p}\wedge 1$ up to constants. As discussed earlier, this matches the result of \citet{isotonic_general_dimensions} and \citet{deng_and_zhang_2020}.

Let us next consider the $p=2$ case. Since for any $\varepsilon>0$ satisfying $n\varepsilon^2 \le \log \mathcal{M}_{\cF_2}^{\loc}(\varepsilon,c)$ we have $n\varepsilon^2\lesssim \varepsilon^{-2}(\log 1/\varepsilon)^2$, it follows that $\varepsilon^{\ast}$ as defined in Theorem \ref{theorem:bounded:minimax:achieved} satisfies $n{\varepsilon^{\ast}}^4\lesssim(\log 1/\varepsilon^{\ast})^2.$ On the other hand, since $\log \mathcal{M}_{\cF_2}^{\loc}(\varepsilon,c)\gtrsim\varepsilon^{-2}$, we have $n{\varepsilon^{\ast}}^2\gtrsim {\varepsilon^{\ast}}^{-2}$, i.e., $\varepsilon^{\ast}\gtrsim n^{-1/4}$. Hence \[n^{-1/4} \lesssim \varepsilon^{\ast} \lesssim n^{-1/4}\sqrt{\log(1/\varepsilon^{\ast})}.\] Reusing $n^{-1/4} \lesssim \varepsilon^{\ast}$ note that $\sqrt{\log(1/\varepsilon^{\ast})} \lesssim\sqrt{\log(n^{1/4})} \asymp \sqrt{\log n}.$ Hence $n^{-1/4} \lesssim \varepsilon^{\ast} \lesssim n^{-1/4}\sqrt{\log(n)}.$ The minimax rate therefore satisfies 
\[n^{-1/2} \wedge 1 \lesssim {\varepsilon^{\ast}}^2 \wedge 1 \lesssim (n^{-1/2}\log n)\wedge 1.\]

For $p=1$, we know $\log \cM(\varepsilon,\cF_1)\asymp \varepsilon^{-1}$ by \citet[Theorem 2.7.5]{van1996weak}. Reusing the result from \citet{yang1999information} we conclude $\log \cM_{\cF_1}^{\loc}(\varepsilon,c)\asymp \varepsilon^{-1}$, and this rate matches $n\varepsilon^2$ when $\varepsilon\asymp n^{-1/3}$. The minimax rate is therefore $n^{-2/3}\wedge 1$.

\subsubsection{Adaptive Rates}

Adaptive rates for monotone functions that are piecewise constant on hyper-rectangles have been studied in \citet{isotonic_general_dimensions} with the least squares estimator (LSE) and also in \citet{deng_and_zhang_2020} with a block estimator. We now demonstrate that our algorithm recovers the adaptive rates given by the LSE in \citet[Theorem 5]{isotonic_general_dimensions}. Let $p\ge 2$ and suppose $f^{\ast}\in\mathcal{F}_p$ is piecewise constant on $k$ hyper-rectangles as defined in Section 3 of \cite{isotonic_general_dimensions}. The convexity of $\mathcal{F}_p$ means we can apply our adaptivity result in Theorem \ref{theorem:nonparam:upperbound:main:adaptive}. Choose $\varepsilon$ such that $n\varepsilon^2 \gtrsim \log \mathcal{M}_{\mathcal{F}_p}^{\adloc}(f^{\ast},\kappa\varepsilon,4c)$  for some suitable constant $\kappa$ chosen from Theorem \ref{theorem:nonparam:upperbound:main:adaptive}. In particular, $\varepsilon_{J^{\ast}}$ from Theorem \ref{theorem:nonparam:upperbound:main:adaptive} satisfies this condition.  Assume $f_1,\dots,f_M\in\mathcal{F}_p$ form a packing set with $ M = \mathcal{M}_{\mathcal{F}_p}^{\adloc}(f^{\ast},\kappa\varepsilon,4c)$. This means $\|f_i-f_j\|_{L_2(\PP_X)}>\kappa\varepsilon/(4c)$ for all $i,j\in[M]$ with $i\ne j$ and $\|f_i-f^{\ast}\|_{L_2(\PP_X)} \le \kappa\varepsilon$ for all $i\in[M]$.

Now, in the proof of Lemma \ref{lemma:bernstein:nonparam:v2}, we obtain two separate bounds that relate the two norms $L_2(\PP_X)$ and $L_2(\PP_n)$. Applying one of the bounds, we set $\delta = \sqrt{n}\kappa\varepsilon$ and note for each $i\in[M]$ we have $n\|f_i-f^{\ast}\|_{L_2(\PP_n)}^2 \le 2n\kappa^2\varepsilon^2$ with probability $1-\exp\left(-\Omega(n\varepsilon^2)\right)$. Here $\Omega(\cdot)$ indicates absolute, positive multiplicative constants. By a union bound, this holds for all $i\in[M]$ simultaneously with probability $1-M\exp\left(-\Omega(n\varepsilon^2)\right)$. Similarly, taking $C$ to be some large constant and $\delta = \sqrt{n}\kappa\varepsilon/(4Cc)$, we have for any $i,j\in[M]$ with $i\ne j$ that $n\|f_i-f_j\|_{L_2(\PP_n)}^2 \ge n(C^2-1)\kappa^2\varepsilon^2/(16C^2c^2)$ with probability $1-\exp\left(-\Omega(n\varepsilon^2)\right)$. Again by a union bound, this holds for all $i,j\in[M]$ with $i\ne j$ simultaneously with probability $1-M^2\exp\left(-\Omega(n\varepsilon^2)\right)$. 

Now let $S$ be the event that $(X_1,\dots,X_n)$ simultaneously satisfies the conditions on $n\|f_i-f_j\|_{L_2(\PP_n)}^2$  for all $i,j\in[M]$ with $i\ne j$ and $n\|f_i-f^{\ast}\|_{L_2(\PP_n)}^2$ for all $i\in[M]$. We have shown that $\PP(S)\ge 1-2M^2\exp(-\Omega(n\varepsilon^2))$, and by our assumption on $\varepsilon$, $\PP(S)\ge \eta>0$ for some constant $\eta>0$. Let \[Q(\vec{X}):= \{(f(X_1),\dots,f(X_n)):f\in\mathcal{F}_p\}\subset\mathbb{R}^n\]
be the set of evaluations of functions in $\mathcal{F}_p$ on our random design points. Additionally, recall that we can write $n\|f-g\|_{L_2(\PP_n)}^2= \|f(\vec{X})-g(\vec{X})\|_2^2$ using our shorthand notation where $f(\vec{X})$ is the vector in $\RR^n$ with $i$th coordinate $f(X_i)$. Clearly each $f_i(\vec{X})\in Q(\vec{X})$. Abbreviate $\tilde c = \sqrt{\frac{32C^2c^2}{C^2-1}}$. Then with probability $\PP(S)$, \[\log \mathcal{M}_{\mathcal{F}_p}^{\adloc}(f^{\ast},\kappa\varepsilon,4c) \le \log \mathcal{M}_{Q(\vec{X})}^{\adloc}\left(f^{\ast}(\vec{X}),\sqrt{2n}\kappa\varepsilon, \tilde c\right),\] noting that the right-hand side is the adaptive local metric entropy of a random subset of $\RR^n$.

We now mimic the argument from the sparse vectors in the $\ell_1$-ball example using Sudakov minoration to relate the metric entropy to the Gaussian width $w$ of a (random) tangent cone, defined by $T(\mu;Q(\vec{X})) = \{t(\nu-\mu):\nu\in Q(\vec{X}), t\ge 0\}$. Note that  \[\frac{B_2^n(f^{\ast}(\vec{X}), \sqrt{2n}\kappa\varepsilon)\cap Q(\vec{X}) - f^{\ast}(\vec{X})}{\sqrt{2n}\kappa\epsilon}\subseteq  T(f^{\ast}(\vec{X});Q(\vec{X}))\cap B_2^n(0,1).\] Observe that \begin{align*}
    \log \mathcal{M}_{Q(\vec{X})}^{\adloc}\left(f^{\ast}(\vec{X}),\sqrt{2n}\kappa\varepsilon,\tilde c\right) &= \log \mathcal{M}\left(\sqrt{2n}\kappa\varepsilon/\tilde c,B_2^n(f^{\ast}(\vec{X}), \sqrt{2n}\kappa\varepsilon)\cap Q(\vec{X})\right) \\
    &\lesssim \frac{1}{n\kappa^2\varepsilon^2}\cdot  \left[w(B_2^n(f^{\ast}(\vec{X}), \sqrt{2n}\kappa\varepsilon)\cap Q(\vec{X}))\right]^2 \\
    &\asymp  \left[w\left( \frac{B_2^n(f^{\ast}(\vec{X}), \sqrt{2n}\kappa\varepsilon)\cap Q(\vec{X}) - f^{\ast}(\vec{X})}{\sqrt{2n}\kappa\varepsilon}  \right)\right]^2\\ 
    &\le  \left[w(T(f^{\ast}(\vec{X});Q(\vec{X}))\cap B_2^n(0,1))\right]^2.
\end{align*} We then relate the Gaussian width of the tangent cone to its statistical dimension $\delta$  (defined in \eqref{eq:statistical:dimension}) using Jensen's inequality:
\begin{align*}
\MoveEqLeft
    \left[w(T(f^{\ast}(\vec{X});Q(\vec{X}))\cap B_2^n(0,1))\right]^2 \\ &=  \left(\EE_{g\sim \cN(0,\mathbb{I}_n)}\left[\sup_{t \in T(f^{\ast}(\vec{X});Q(\vec{X}))\cap B_2^n(0,1)}\langle t,g\rangle \right]\right)^2 \\
    &\le  \EE_{g\sim \cN(0,\mathbb{I}_n)}\left[\left(\sup_{t \in T(f^{\ast}(\vec{X});Q(\vec{X}))\cap B_2^n(0,1)}\langle t,g\rangle\right)^2 \right] \\
    &=  \delta(T(f^{\ast}(\vec{X});Q(\vec{X}))).
\end{align*} 

We can thus bound the entropy with a conditional expectation and then apply the law of total expectation. After that we use \citet[Equation (21)]{isotonic_general_dimensions}, which bounds the statistical dimension term up to multiplicative constants depending on $p$ by $n \cdot(k/n)^{2/p}\cdot \log_{+}^{2\gamma_p}(n/k)$, where the authors define $\gamma_p=(p^2+p+1)/2$ if $p\ge 3$,  $\gamma_2=9/2$, and $\log_{+}(a):=\log(a\vee e)$. Thus, we have \begin{align*}
    \log \mathcal{M}_{\mathcal{F}_p}^{\adloc}(f^{\ast},\kappa\varepsilon,4c) &\lesssim \EE_{X}[\delta(T(f^{\ast}(\vec{X});Q(\vec{X})) | X_1,\dots,X_n\in S] \\
    & \leq \EE_{X}[\delta(T(f^{\ast}(\vec{X});Q(\vec{X}))]/\PP(S)\\
    &\lesssim n\cdot (k/n)^{2/p}\cdot\log_{+}^{2\gamma_p}(n/k),
\end{align*}
where we used that $\PP(S)\ge \eta>0$ for some constant $\eta>0$ as we argued earlier from our assumption on $\varepsilon$. 
Therefore, applying Theorem \ref{theorem:nonparam:upperbound:main:adaptive}, using that $\varepsilon=\varepsilon_{J^{\ast}}$ satisfies $n\varepsilon^2 \gtrsim \log \mathcal{M}_{\mathcal{F}_p}^{\adloc}(f^{\ast},\kappa\varepsilon,4c)$, and that $J^{\ast}$ is chosen maximally, our estimator $f^{\ddagger}$ satisfies \begin{align*}
    \EE_{\vec{X},\vec{Y}}\|f^{\ast} - f^{\ddagger}(\vec{X},\vec{Y})\|_{\LtwoP}^2 &\lesssim \varepsilon_{J^{\ast}}^2 \asymp n^{-1}\left[\log \mathcal{M}_{\mathcal{F}_p}^{\adloc}(f^{\ast},\kappa\varepsilon_{J^{\ast}},4c) \vee \log 2\right] \\ &\lesssim (k/n)^{2/p}\cdot \log_{+}^{2\gamma_p}(n/k).
\end{align*} 
\section{Acknowledgements}
\label{section:nonparam:acknowledgements}

The authors would like to thank Arun Kumar Kuchibhotla, Sivaraman Balakrishnan, Larry Wasserman and Alexandre Tsybakov for discussions which led to significant improvements in the paper---specifically they led to the realization that our technique works for function classes more general than ones bounded in sup norm or with an $L$-sub Gaussian assumption. The authors are also grateful to the several referees and Associate Editors for comments that strengthened the paper.

\bibliographystyle{abbrvnat}
\bibliography{nonparam_ref}

\begin{thebibliography}{26}
\providecommand{\natexlab}[1]{#1}
\providecommand{\url}[1]{\texttt{#1}}
\expandafter\ifx\csname urlstyle\endcsname\relax
  \providecommand{\doi}[1]{doi: #1}\else
  \providecommand{\doi}{doi: \begingroup \urlstyle{rm}\Url}\fi

\bibitem[Amelunxen et~al.(2014)Amelunxen, Lotz, McCoy, and
  Tropp]{amelunxen2014living}
D.~Amelunxen, M.~Lotz, M.~B. McCoy, and J.~A. Tropp.
\newblock Living on the edge: Phase transitions in convex programs with random
  data.
\newblock \emph{Information and Inference: A Journal of the IMA}, 3\penalty0
  (3):\penalty0 224--294, 2014.

\bibitem[Boucheron et~al.(2013)Boucheron, Lugosi, and Massart]{boucheron2013}
S.~Boucheron, G.~Lugosi, and P.~Massart.
\newblock \emph{{Concentration Inequalities: A Nonasymptotic Theory of
  Independence}}.
\newblock Oxford University Press, 02 2013.
\newblock ISBN 9780199535255.
\newblock \doi{10.1093/acprof:oso/9780199535255.001.0001}.

\bibitem[Chandrasekaran et~al.(2012)Chandrasekaran, Recht, Parrilo, and
  Willsky]{chandrasekaran2012convex}
V.~Chandrasekaran, B.~Recht, P.~A. Parrilo, and A.~S. Willsky.
\newblock The convex geometry of linear inverse problems.
\newblock \emph{Foundations of Computational Mathematics}, 12\penalty0
  (6):\penalty0 805–849, October 2012.
\newblock ISSN 1615-3383.
\newblock \doi{10.1007/s10208-012-9135-7}.

\bibitem[Deng and Zhang(2020)]{deng_and_zhang_2020}
H.~Deng and C.-H. Zhang.
\newblock {Isotonic regression in multi-dimensional spaces and graphs}.
\newblock \emph{The Annals of Statistics}, 48\penalty0 (6):\penalty0 3672 --
  3698, 2020.
\newblock \doi{10.1214/20-AOS1947}.
\newblock URL \url{https://doi.org/10.1214/20-AOS1947}.

\bibitem[Gao and Wellner(2007)]{gao_Wellner_monotone}
F.~Gao and J.~A. Wellner.
\newblock Entropy estimate for high-dimensional monotonic functions.
\newblock \emph{Journal of Multivariate Analysis}, 98\penalty0 (9):\penalty0
  1751--1764, 2007.
\newblock ISSN 0047-259X.
\newblock \doi{https://doi.org/10.1016/j.jmva.2006.09.003}.

\bibitem[Györfi(2002)]{gyorfi_2022}
L.~Györfi.
\newblock \emph{A Distribution-Free Theory of Nonparametric Regression}.
\newblock Springer Series in Statistics. Springer, 2002.
\newblock ISBN 978-0-387-95441-7.

\bibitem[Han and Wellner(2019)]{qiyang_wellner}
Q.~Han and J.~A. Wellner.
\newblock {Convergence rates of least squares regression estimators with
  heavy-tailed errors}.
\newblock \emph{The Annals of Statistics}, 47\penalty0 (4):\penalty0 2286 --
  2319, 2019.
\newblock \doi{10.1214/18-AOS1748}.
\newblock URL \url{https://doi.org/10.1214/18-AOS1748}.

\bibitem[Han et~al.(2019)Han, Wang, Chatterjee, and
  Samworth]{isotonic_general_dimensions}
Q.~Han, T.~Wang, S.~Chatterjee, and R.~J. Samworth.
\newblock {Isotonic regression in general dimensions}.
\newblock \emph{The Annals of Statistics}, 47\penalty0 (5):\penalty0 2440 --
  2471, 2019.
\newblock \doi{10.1214/18-AOS1753}.

\bibitem[Mendelson(2017)]{mendelson_2017_local_versus_global}
S.~Mendelson.
\newblock {“Local” vs. “global” parameters—breaking the Gaussian
  complexity barrier}.
\newblock \emph{The Annals of Statistics}, 45\penalty0 (5):\penalty0 1835 --
  1862, 2017.
\newblock \doi{10.1214/16-AOS1510}.

\bibitem[Neykov(2022)]{neykov2022minimax}
M.~Neykov.
\newblock On the minimax rate of the gaussian sequence model under bounded
  convex constraints.
\newblock \emph{IEEE Transactions on Information Theory}, 69\penalty0
  (2):\penalty0 1244--1260, 2022.
\newblock \doi{10.1109/TIT.2022.3213141}.

\bibitem[Pathak et~al.(2024)Pathak, Wainwright, and Xiao]{pathak2023noisy}
R.~Pathak, M.~J. Wainwright, and L.~Xiao.
\newblock {Noisy recovery from random linear observations: Sharp minimax rates
  under elliptical constraints}.
\newblock \emph{The Annals of Statistics}, 52\penalty0 (6):\penalty0 2816 --
  2850, 2024.
\newblock \doi{10.1214/24-AOS2446}.
\newblock URL \url{https://doi.org/10.1214/24-AOS2446}.

\bibitem[Prasadan and
  Neykov(2024)]{prasadan2024informationtheoreticlimitsrobust}
A.~Prasadan and M.~Neykov.
\newblock Information theoretic limits of robust sub-gaussian mean estimation
  under star-shaped constraints.
\newblock \emph{arXiv preprint arXiv:2412.03832}, 2024.

\bibitem[Raskutti et~al.(2011)Raskutti, Wainwright, and
  Yu]{raskutti2011minimax}
G.~Raskutti, M.~J. Wainwright, and B.~Yu.
\newblock Minimax rates of estimation for high-dimensional linear regression
  over $\ell_q $-balls.
\newblock \emph{IEEE transactions on information theory}, 57\penalty0
  (10):\penalty0 6976--6994, 2011.

\bibitem[Robbins(1955)]{robbins1955}
H.~Robbins.
\newblock A remark on stirling's formula.
\newblock \emph{The American Mathematical Monthly}, 62\penalty0 (1):\penalty0
  26--29, 1955.
\newblock ISSN 00029890, 19300972.
\newblock URL \url{http://www.jstor.org/stable/2308012}.

\bibitem[Shrotriya and Neykov(2022)]{lecamshamindramatey2022}
S.~Shrotriya and M.~Neykov.
\newblock Revisiting le cam's equation: Exact minimax rates over convex density
  classes.
\newblock \emph{arXiv preprint arXiv:2210.11436}, 2022.

\bibitem[Stone(1980)]{stone_1980}
C.~J. Stone.
\newblock {Optimal Rates of Convergence for Nonparametric Estimators}.
\newblock \emph{The Annals of Statistics}, 8\penalty0 (6):\penalty0 1348 --
  1360, 1980.
\newblock \doi{10.1214/aos/1176345206}.

\bibitem[Stone(1982)]{stone_1982}
C.~J. Stone.
\newblock {Optimal Global Rates of Convergence for Nonparametric Regression}.
\newblock \emph{The Annals of Statistics}, 10\penalty0 (4):\penalty0 1040 --
  1053, 1982.
\newblock \doi{10.1214/aos/1176345969}.
\newblock URL \url{https://doi.org/10.1214/aos/1176345969}.

\bibitem[Tsybakov(2009)]{tsybakov2009introduction}
A.~B. Tsybakov.
\newblock \emph{Introduction to Nonparametric Estimation.}
\newblock Springer, 2009.

\bibitem[Van Der~Vaart and Wellner(1996)]{van1996weak}
A.~W. Van Der~Vaart and J.~Wellner.
\newblock \emph{Weak convergence and empirical processes: with applications to
  statistics}.
\newblock Springer Science \& Business Media, 1996.

\bibitem[Vershynin(2018)]{vershynin2018high}
R.~Vershynin.
\newblock \emph{High-dimensional probability: An introduction with applications
  in data science}, volume~47.
\newblock Cambridge University Press, 2018.

\bibitem[Wainwright(2019)]{wainwright2019high}
M.~J. Wainwright.
\newblock \emph{High-dimensional statistics: A non-asymptotic viewpoint},
  volume~48.
\newblock Cambridge University Press, 2019.

\bibitem[Wang et~al.(2014)Wang, Paterlini, Gao, and Yang]{wang_yang_2014}
Z.~Wang, S.~Paterlini, F.~Gao, and Y.~Yang.
\newblock Adaptive minimax regression estimation over sparse $\ell_q$-hulls.
\newblock \emph{Journal of Machine Learning Research}, 15\penalty0
  (50):\penalty0 1675--1711, 2014.
\newblock URL \url{http://jmlr.org/papers/v15/wang14b.html}.

\bibitem[Yang(2004)]{agglomeration_yang_2004}
Y.~Yang.
\newblock Aggregating regression procedures to improve performance.
\newblock \emph{Bernoulli}, 10\penalty0 (1):\penalty0 25--47, 2004.
\newblock ISSN 13507265.
\newblock URL \url{http://www.jstor.org/stable/3318829}.

\bibitem[Yang and Barron(1999)]{yang1999information}
Y.~Yang and A.~Barron.
\newblock Information-theoretic determination of minimax rates of convergence.
\newblock \emph{Annals of Statistics}, pages 1564--1599, 1999.

\bibitem[Yang and Tokdar(2015)]{yang_2014_high_dimensional_nonparam}
Y.~Yang and S.~T. Tokdar.
\newblock {Minimax-optimal nonparametric regression in high dimensions}.
\newblock \emph{The Annals of Statistics}, 43\penalty0 (2):\penalty0 652 --
  674, 2015.
\newblock \doi{10.1214/14-AOS1289}.

\bibitem[Zhao and Yang(2023)]{zhao2023minimaxratesconvergencenonparametric}
B.~Zhao and Y.~Yang.
\newblock Minimax rates of convergence for nonparametric location-scale models.
\newblock \emph{arXiv preprint arXiv:2307.01399}, 2023.

\end{thebibliography}

\clearpage

\appendix

\section{Proofs for Section \ref{section:nonparam:lower:bound}}

\begin{proof}[\hypertarget{proof:lemma:nonparam:mutualinfo}{Proof of Lemma \ref{lemma:nonparam:mutualinfo}}] 
    We use \cite[(15.52)]{wainwright2019high}, the decoupling property of Kullback-Leibler (KL) divergence $D_{\mathrm{KL}}$, and the decomposition of the joint distribution of $(X_i,Y_i)$. 
    Let $\PP_{Y_i|X_i}^{f_j}$ be the distribution of $(Y_i |X_i, J = j)$, which is a normal distribution with mean $f_j(X_i)$ and variance $\sigma^2$. Denote by $\PP_{(X_i,Y_i)}^{f_j}$ the distribution of $(Y_i,X_i|J=j)$ and $\PP_{X_i}$ that of $X_i$, so that $\PP_{(X_i,Y_i)}^{f_j} = \PP_{X_i}\times \PP_{Y_i|X_i}^{f_j}$. Note that the $X_i$ does not depend on $f_j$. Similarly, for any function $h\in\cF$, if we set $Y_i=h(X_i)+\xi_i$ instead, we denote the analogous distributions by $\PP_{(X_i,Y_i)}^{h} = \PP_{X_i}\times \PP_{Y_i|X_i}^{h}$.
    Then we obtain 
    \begin{align*}
        I(\vec{X},\vec{Y};J) 
        &\le \frac{1}{m} \sum_j D_{\mathrm{KL}}\left(\prod_i\PP_{(X_i,Y_i)}^{f_j} \parallel \prod_i       \PP_{(X_i,Y_i)}^{h}\right) \\
        &= \frac{1}{m} \sum_j \sum_i D_{\mathrm{KL}}\left(\PP_{(X_i,Y_i)}^{f_j} 
            \parallel \PP_{(X_i,Y_i)}^{h}\right) \\
        &=  \frac{1}{m} \sum_j \sum_i \EE_{f_j}\left[\log\left(
            \PP_{(X_i,Y_i)}^{f_j}/ \PP_{(X_i,Y_i)}^{h}\right)\right] \\
        &= \frac{1}{m} \sum_j \sum_i\EE_{f_j}
            \left[\log\left(\frac{\PP_{X_i}\PP_{(Y_i|X_i)}^{f_j}}{\PP_{X_i}\PP_{(Y_i|X_i)}^{h}}\right)\right] \\
        &= \frac{1}{m} \sum_j \sum_i\EE_{X_i} \EE_{Y_i | (X_i, f_j)}
            \left[\log\left(\PP_{(Y_i|X_i)}^{f_j}/
            \PP_{(Y_i|X_i)}^{h}\right)\right] \\
        &= \frac{1}{m} \sum_j \sum_i\EE_{X_i} 
            \left[D_{\mathrm{KL}}\left(\cN(f_j(X_i), \sigma^2) 
            \parallel \cN(h(X_i), \sigma^2)\right)\right] \\
        &= \frac{1}{m} \sum_j \sum_i \EE_{X_i}
            \left[\frac{|f_j(X_i)-h(X_i)|^2}{2\sigma^2}\right] \\
        &= \frac{1}{2\sigma^2 m}\sum_j \sum_i \|f_j-h\|_{\LtwoP}^2 \\
        &= \frac{n}{2\sigma^2 m}\sum_j \|f_j-h\|_{\LtwoP}^2 \\
        &\le \frac{n}{2\sigma^2}\max_j \|f_j-h\|_{\LtwoP}^2.
\end{align*}
\end{proof}

\begin{proof}[\hypertarget{proof:lemma:nonparam:lowerbound}{Proof of Lemma \ref{lemma:nonparam:lowerbound}}] 
    Pick a $f\in\cF$ that achieves \[M:=\cMFloc(\varepsilon, c):=\sup_{f \in \cF} \cM(\varepsilon/c, B(f, \varepsilon) \cap \cF)\] (otherwise, repeat the following argument by picking a limiting sequence of functions). Let $f_1,\dots,f_M$ be a maximal $\varepsilon/c$-packing set for $B(f, \varepsilon) \cap \cF$, so that $\|f_i-f_j\|_{\LtwoP}>\varepsilon/c$ for all $i\ne j$. Because $f_1,\dots,f_M\in B(f, \varepsilon) \cap \cF$, it follows for each $j\in[M]$ that $\|f_j-f\|_{\LtwoP} \le \varepsilon$. Letting $J$ be uniformly distributed on $[M]$, we have from Lemma \ref{lemma:nonparam:mutualinfo} that \begin{align*}
        I(\vec{X},\vec{Y};J) &\le \frac{n}{2\sigma^2}\max_j  \|f_j-f\|_{\LtwoP}^2 \le \frac{n\varepsilon^2}{2\sigma^2}.
    \end{align*} Applying Lemma \ref{lemma:fano:nonparam}, we obtain \begin{align*}
        \inf_{\hat{f}}\sup_{\bar{f}\in\cF}\EE_{\vec{X},\vec{Y}}\|\hat{f}(\vec{X},\vec{Y})-\bar{f}\|_{\LtwoP}^2 &\geq \frac{\varepsilon^2}{4c^2}\bigg(1 - \frac{\frac{n\varepsilon^2}{2\sigma^2} + \log 2}{\log M}\bigg).
    \end{align*}
    Choose $\varepsilon$ such that $\log M$ is bigger than the middle term in \[\log M\ge 4\left(\frac{n\varepsilon^2}{2\sigma^2} \vee \log 2\right) \ge 2\left(\frac{n\varepsilon^2}{2\sigma^2}+\log 2\right),\] and this proves the minimax risk is lower-bounded by $\frac{\varepsilon^2}{8c^2}$.
\end{proof}

\section{Proofs for Section \ref{section:nonparam:upper:bound}}

\begin{remark}[Breaking Ties in Algorithm] \label{remark:breaking_ties}
        By Theorem A.1 in \citet{gyorfi_2022}, the set of continuous functions is dense in $\cF$. Hence we can construct all of our packing sets using continuous functions only. That is, if $\cF_c$ denote the set of all continuous functions in $\cF$, then we construct maximal packing sets $M_k$ of sets $B(\cdot, d/2^{k-1}) \cap \cF_c$ prior to pruning. This $M_k$ clearly has cardinality that is smaller than or equal to the $M_k$ constructed with $\cF$ instead since $\cF_c \subseteq \cF$. In addition, since the set of all continuous functions is dense in $\cF$, for any $g\in \cF$ and $\varepsilon>0$, there exists an $h \in \cF_c$ such that $\|g-h\|_{\LtwoP}\leq \varepsilon$. Pick the point $f \in M_k$ which is $d/2^{k-1}$ close to $h$ and conclude that $\|f- g\|_{\LtwoP} \leq \varepsilon + d/2^{k-1}$. Let $\varepsilon$ go to zero, and we conclude that any point in $\cF$ can be approximated by a point in $\cF_c$ within distance $d/2^{k-1}$. Now suppose that the domain of the functions is $\RR^p$. Take the set $\QQ^p$ which is dense in $\RR^p$. This is a well-ordered set since there is a bijection between $\QQ^p$ and $\NN$. Hence the functions in $M_k$ can be restricted to $\QQ^p$ and then ordered lexicographically. There will not be any ties, since the functions are separated and continuous (hence the functions cannot be the same on $\QQ^p$ otherwise they are the same on $\RR^p$ which is a contradiction) so we can take the smallest function in the lexicographic ordering.
\end{remark}

    \begin{proof}[Proof of Proposition \ref{proposition:L:subGaussian:implies:moment:condition}]
        A uniform sup-norm bound trivially implies \eqref{new:bernstein:sufficient:condition}, so we just consider the $L$-sub-Gaussian case. Also note that by Stirling's approximation \citep{robbins1955}, we have \begin{align*}
            p^p &\le \frac{p!\exp(p)\exp\left(-\tfrac{1}{12p+1}\right)}{\sqrt{2\pi p}} = \frac{ p! \cdot\exp\left(p - \frac{1}{12p + 1}-\frac{1}{2}\log p\right)}{\sqrt{2\pi}} \\
            &=  \frac{p!}{\sqrt{2\pi}}\cdot\bigg[\exp\underbrace{\left(1 - \tfrac{1}{p(12p + 1)}-\tfrac{1}{2p}\log p\right)}_{\le 1}\bigg]^p \le \tilde C e^p p!.
        \end{align*}
        Therefore \begin{align*}
            \|f - g\|_{L_{2p}(\PP_X)}^{2p} &\le \left(L\sqrt{2p} \|f - g\|_{\LtwoP}\right)^{2p} \\ &= (2L^2)^p p^p \left(\|f - g\|_{\LtwoP}^2\right)^p \\
            &= (2L^2)^p p^p \left(\|f - g\|_{\LtwoP}^2\right)^{p-2}\|f - g\|_{\LtwoP}^2\cdot  \\
            &\le (2L^2)^p d^{p-2} p^p\|f - g\|_{\LtwoP}^2 \\
            &\le (2L^2 )^p d^{p-2} \tilde C e^p p! \|f - g\|_{\LtwoP}^2 \\
            &= \frac{\tilde{L}^{p-1}p!\|f - g\|_{\LtwoP}^2}{2},
        \end{align*} where  $\tilde L = (2\tilde C)^{\tfrac{1}{p-1}}\cdot (2L^2 e)^{\tfrac{p}{p-1}} d^{\tfrac{p-2}{p-1}} $. Note that $\tfrac{1}{p-1}\le 1$, $\tfrac{p}{p-1}\le 2$, and $\tfrac{p-2}{p-1}\le 1$. Hence \[\tilde L \le (2\tilde C)(2L^2e)^2 d=8e^2\tilde CL^4d.\] Thus, setting $L' = 8e^2\tilde CL^4d$, we have $\|f - g\|_{L_{2p}(\PP_X)}^{2p} \le \|f - g\|_{\LtwoP}^2 {L'}^{p-1}p!/2$ for some fixed constant $L'>0$ and all $p\ge 2$.
    \end{proof}

\begin{proof}[Proof of Example \ref{new:bernstein:sufficient:condition}] 
         Observe that any difference $h=f-g$ of two functions $f,g$ in $\cF$ satisfies $h(0) = 0$ and is $2$-Lipschitz. We need to argue that for some $L$, for all $p\ge 2$ we have
\begin{align}\label{lipschitz:identity:function:class}
    \EE h^{2p}(X) \leq \EE h^{2}(X) L^{p-1} p!/2.
\end{align}
By Rademacher's theorem, $h$ is differentiable almost everywhere, and since $h$ is $2$-Lipschitz we have $|h'(x)| \leq 2$ at any point of differentiability $x$. In addition, $|h(x)| \leq 2 |x|$. Hence by Stein's Lemma applied to the function $\phi(x)=h^{2p}(x)/x$, we have
\begin{align*}
     \EE h^{2p}(X) & = \EE \left[\frac{h^{2p}(X)\cdot X}{X}\right]= 2p\cdot  \EE \left[\frac{h^{2p-1}(X) h'(X)}{X}\right] - \EE \left[\frac{h^{2p}(X)}{X^2}\right]  \\
     & \leq 8p \EE\left[ \frac{|h^{2p-1}(X)|}{2 |X|}\right]\leq  8p \EE [h^{2p-2}(X)].
\end{align*}
Repeating this logic recursively $p-2$ times, 
\begin{align*}
     \EE h^{2p}(X)\leq 8^{p-1} p! \EE h^2(X). 
\end{align*} 
Thus \eqref{lipschitz:identity:function:class} holds with $L=2^{(3p-2)/(p-1)} \le 2^4 = 16$ since $\frac{3p-2}{p-1}\le 4$ for any $p\ge 2$. Finally we argue that the class $\cF$ has a bounded $\LtwoP$ diameter. This is so since $
    \EE h^{2}(X) \leq 4 \EE X^2 = 4.$

However, this function class is clearly not bounded in sup-norm, and we can also show it is not $L$-sub-Gaussian. To see the latter, we follow the logic of the counter-examples in \citet{qiyang_wellner}, specifically Proposition 3 where the authors show a small-ball condition defined in their Section 5.2.1 fails.\footnote{The proposition and section referred to correspond to an earlier version of a preprint for this work, namely arXiv:1706.02410v1.} By the Paley-Zygmund inequality, this will also imply the failure of $L$-sub-Gaussianity. Define the sequence of functions $\{f_n\}_{n\in\NN}$ as follows. Let $f_n(x)$ be $0$ everywhere except let its graph form a triangle with vertices at $(n,0)$, $(n+2,0)$, and $(n+1,1)$, or more formally, \begin{align*}
    f_n(x) = \begin{cases}
        0 & x < n \text{ or }x \ge n+2 \\
        x-n & x\in[n, n+1) \\
        n+2-x & x\in[n+1, n+2).
    \end{cases}
\end{align*}
This clearly satisfies $f_n(0)=0$ and the $1$-Lipschitz requirement (hence it can be represented as a difference between a $1$-Lipschitz function and the zero function (which is also $1$-Lipschitz)). Now pick any $\eta>0$. Observe that \[\PP(|f_n(X)| \ge \eta \|f_n\|_{\LtwoP})\le \PP(|f_n(X)|>0)\leq\PP(X\in(n,n+2))\to 0,\] the last step using that $X\sim \cN(0,1)$. This implies the small-ball condition fails. Hence $\cF$ is not $L$-sub-Gaussian. 
    \end{proof}

\begin{proof}[Proof of Lemma \ref{lemma:moment:condition:implies:bernstein:concentration}] 
    For $\theta>0$, we multiply the below inequality inside the probability by $\theta$ and exponentiate. Then an application of Markov's inequality yields
    \begin{align}
        \PP\bigg(\sum_{i\in[n]} (Z_i^2 - \EE Z^2) \geq t\bigg) &\leq  \frac{1}{\exp(\theta t)}\cdot \EE \exp\bigg(\theta \sum_{i\in[n]} (Z_i^2 - \EE Z^2)\bigg) \notag \\ &= \frac{1}{\exp(\theta t)}\cdot\prod_{i \in [n]}\EE \exp\left(\theta (Z_i^2 - \EE Z^2)\right). \label{eq:markov:to:bernstein}
    \end{align}
     Similarly, still assuming $\theta>0$ but this time multiplying by $-\theta$ and exponentiating, we have
        \begin{align*}
        \PP\bigg(\sum_{i\in[n]} (Z_i^2 - \EE Z^2\bigg) \leq -t) &\leq \frac{1}{\exp(\theta t)}\cdot\EE \exp\bigg(-\theta\sum_{i\in[n]} (Z_i^2 - \EE Z^2)\bigg) \\ &= \frac{1}{\exp(\theta t)}\cdot\prod_{i \in [n]}\EE \exp\big(-\theta (Z_i^2 - \EE Z^2)\big).
    \end{align*}

    Now we control each $\EE \exp(\eta \theta (Z_i^2 - \EE Z^2))$ from above, where $\eta \in \{\pm 1\}$. We have
    \begin{align*}
        \EE \exp(\eta \theta (Z_i^2 - \EE Z^2)) & = 1 + \sum_{i = 2}^\infty \frac{\theta^i \EE (\eta(Z^2 - \EE Z^2))^i}{i!} \\
        & \leq 1 + \sum_{i = 2}^\infty \frac{\theta^i \EE (|Z^2 - \EE Z^2|)^i}{i!} \\
        & \leq 1 + \sum_{i = 2}^\infty \frac{\theta^i \EE (|Z^2| + |\EE Z^2|)^i}{i!}\\ 
        & \leq 1 + \sum_{i = 2}^\infty \frac{\theta^i 2^{i-1}(\EE Z^{2i} + (\EE Z^2)^i)}{i!}\\
        & \leq 1 +\sum_{i = 2}^\infty \frac{\theta^i 2^{i}\EE Z^{2i}}{i!},
    \end{align*}
    where we used $(\EE Z^2)^i \leq \EE Z^{2i} $ by Jensen's inequality, and $(a+b)^i \leq 2^{i-1} (a^i + b^i)$ for $a,b > 0$ (which follows from Jensen's inequality or induction). Now applying our hypothesis on the $2i$-th moment and recognizing a geometric sum, we have
    \begin{align}
        \EE \exp(\theta \eta (Z_i^2 - \EE Z^2)) & \leq 1 + \sum_{i = 2}^\infty \frac{(2\theta L)^i \EE Z^2}{2 L } \notag\\
        &= 1 +  (2\theta L)^2 \EE Z^2  \sum_{i = 0}^\infty \frac{(2\theta L)^i}{2 L } \notag \\
        & = 1 + \frac{(2\theta)^2 \EE Z^2 L}{2} \frac{1}{1 - 2\theta L} \notag\\
        & \leq \exp\left((2\theta)^2 \EE Z^2 L\cdot \frac{1}{1 - 2\theta L}\right), \label{eq:markov:to:bernstein:part2}
    \end{align}
    provided that $2\theta L < 1$.  The final inequality used the result $1+x \le e^x$ for all $x\in\RR$.

    Next, select $\theta = \frac{t}{8 n \EE Z^2 L + 2 L t}$, so that we indeed have $2 L \theta < 1$. Thus $\tfrac{1}{1- 2\theta L} = \frac{8 n\EE Z^2 L + 2 L t}{8 n \EE Z^2 L}$. Combining \eqref{eq:markov:to:bernstein} and \eqref{eq:markov:to:bernstein:part2},
    \begin{align*}
         \MoveEqLeft \PP(\sum (Z_i^2 - \EE Z^2) \geq t)  \\ &\leq\exp\left((2\theta)^2 \EE Z^2 L \cdot \frac{1}{1 - 2\theta L} - \theta t\right)\\
          &= \exp\left(\frac{(2t)^2 \EE Z^2 L}{(8 n\EE Z^2 L + 2 L t)^2} \frac{8 n\EE Z^2 L + 2 L t}{8 n\EE Z^2 L} - \frac{t^2}{8 n\EE Z^2 L + 2 L t}\right) \\
          &= \exp\left(\frac{4t^2 }{8n\cdot (8 n\EE Z^2 L + 2 L t)}  - \frac{t^2}{8 n\EE Z^2 L + 2 L t}\right) \\
          &\le \exp\left(\frac{4t^2 }{8\cdot (8 n\EE Z^2 L + 2 L t)}  - \frac{t^2}{8 n\EE Z^2 L + 2 L t}\right) \\
          & = \exp\left(-\frac{t^2}{16 n\EE Z^2 L + 4 Lt}\right),
    \end{align*}
    which is what we aimed to show. Finally note that the same bound applies for the other tail.
\end{proof}

\begin{proof}[\hypertarget{proof:lemma:bernstein:nonparam:v2}{Proof of Lemma \ref{lemma:bernstein:nonparam:v2}}]

Suppose we take $Z_i=|f(X_i)-g(X_i)|$ for each $i\in[n]$, noting these are i.i.d. and $\EE Z_i^{2p} \le \EE Z_i^2 L^{p-1}p!/2$ by \eqref{new:bernstein:sufficient:condition}. Note that $\sum_{i=1}^n Z_i^2=\sum_{i=1}^n |f(X_i)-g(X_i)|^2 = n\|f-g\|_{L_2(\PP_n)}^2$ and $\EE[Z_1^2]=\|f-g\|_{L_2(\PP_n)}$. Then by Lemma \ref{lemma:moment:condition:implies:bernstein:concentration} \begin{align}
    \MoveEqLeft \PP\left(n\|f-g\|_{L_2(\PP_n)}^2 - n\|f-g\|_{\LtwoP}^2 \geq t\right)\notag \\ &\vee \PP\left(n\|f-g\|_{L_2(\PP_n)}^2- n\|f-g\|_{\LtwoP}^2 \leq -t\right) \notag \\ &\le \exp\left(-\frac{t^2}{16 nL \|f-g\|_{\LtwoP}^2  + 4 Lt}\right). \label{eq:bernstein:like:bound}
\end{align} An identical bound holds for $\|f-\bar{f}\|_{L_2(\PP_n)}$.

Set $t = C^{-2}\cdot n\|f-g\|_{\LtwoP}^2$. Then we derive the following, where the first and last inequality use $n\|f-g\|_{\LtwoP}^2\ge C^2\delta^2$, and the second inequality uses \eqref{eq:bernstein:like:bound}.
  \begin{align*}
        \MoveEqLeft \PP_{\bar{f}}\left( n\|f-g\|_{L_2(\PP_n)}^2 <(C^2-1)\delta^2 \right) \\ &= \PP_{\bar{f}}\left( n\|f-g\|_{L_2(\PP_n)}^2 < (1-C^{-2})\cdot C^2\delta^2\right)  \\
        &\le  \PP_{\bar{f}}\left( n\|f-g\|_{L_2(\PP_n)}^2 < (1-C^{-2})\cdot n\|f-g\|_{\LtwoP}^2\right) \\
        &= \PP_{\bar{f}}\left( n\|f-g\|_{L_2(\PP_n)}^2 -  n\|f-g\|_{\LtwoP}^2 < \right. \\ &\quad\quad\quad\quad\quad\left. -C^{-2}\cdot n\|f-g\|_{\LtwoP}^2\right) \\
        &\le \exp\left(-\frac{t^2}{16 nL \|f-g\|_{\LtwoP}^2  + 4 Lt}\right)\\
        &=  \exp\left(-\frac{t^2}{16C^2 L t + 4 Lt}\right)\\
        &= \exp\left(-\frac{t}{4L(C^2 + 1)}\right) \\
        &=  \exp\left(-\frac{C^{-2}n\|f-g\|_{\LtwoP}}{4L(C^2 + 1)}\right) \\
        &\le \exp\left(-\frac{\delta^2}{4L(C^2 + 1)}\right).
    \end{align*}

Thus, we have shown that $n\|f-g\|_{L_2(\PP_n)}^2 \ge (C^2-1)\delta^2$ with probability at least \[1-  \exp\left(-\frac{\delta^2}{4L(C^2 + 1)}\right).\]

For the other bound, we now set $Z_i = |f(X_i)-\bar{f}(X_i)|$ for each $i\in[n]$, so that we deduce the same bound \eqref{eq:bernstein:like:bound} with $\|f-\bar{f}\|_{L_2(\PP_n)}$ and $\|f-\bar{f}\|_{\LtwoP}$ instead. Setting $t=\delta^2$ and repeatedly using $n\|f-\bar{f}\|_{\LtwoP}^2<\delta^2$, we have \begin{align*}
    \PP_{\bar{f}}\left( n\|f-\bar{f}\|_{L_2(\PP_n)}^2 > 2\delta^2\right) &\le \PP_{\bar{f}}\left( n\|f-\bar{f}\|_{L_2(\PP_n)}^2 > n\|f-\bar{f}\|_{\LtwoP}^2 + \delta^2\right) \\
    &\le \exp\left(-\frac{t^2}{16 nL \|f-\bar f\|_{\LtwoP}^2  + 4 Lt}\right) \\
    &\le \exp\left(-\frac{t^2}{16 L \delta^2  + 4 Lt}\right) \\
    &= \exp\left(-\frac{\delta^2}{20L}\right).
\end{align*} Thus, we have shown that $n\|f-\bar{f}\|_{L_2(\PP_n)}^2 \le 2\delta^2$ with probability at least \[1- \exp\left(-\frac{\delta^2}{20L}\right).\]

Combining our results with a union bound, with probability at least \[1 -\exp\left(-\frac{\delta^2}{20L}\right)- \exp\left(-\frac{\delta^2}{4L(C^2 + 1)}\right), \] we have both $n\|f-\bar{f}\|_{L_2(\PP_n)}^2 \le 2\delta^2$ and $n\|f-g\|_{L_2(\PP_n)}^2 \ge (C^2-1)\delta^2$.  
\end{proof}

\begin{proof}[\hypertarget{proof:lemma:testing:nonparam:gaussianbound}{Proof of Lemma \ref{lemma:testing:nonparam:gaussianbound}}]

Let $\eta(\vec{X}) =\bar{f}(\vec{X}) - f(\vec{X})$. Then we write
 \begin{align*}
    \MoveEqLeft \|\vec{Y}-f(\vec{X})\|_2^2 - \|\vec{Y}-g(\vec{X})\|_2^2 
        \\ &= \sum_{i=1}^n(Y_i-f(X_i))^2 - \sum_{i=1}^n(Y_i-g(X_i))^2 \\ 
        &= 2\sum_{i=1}^n Y_i(g(X_i)-f(X_i)) + \sum_{i=1}^n f(X_i)^2 - \sum_{i=1}^n g(X_i)^2 \\
        &= 2 \sum_{i=1}^n (\bar{f}(X_i)+\xi_i)(g(x_i)-f(x_i)) + 
            \|f(\vec{X})\|_2^2 - \|g(\vec{X})\|_2^2 \\
        &= 2 (\bar{f}(\vec{X}) +\xi)^T(g(\vec{X})-f(\vec{X})) +
            \|f(\vec{X})\|_2^2 - \|g(\vec{X})\|_2^2 \\
        &= 2 (f(\vec{X})+\eta(\vec{X}) +\xi)^T(g(\vec{X})-f(\vec{X})) +
            \|f(\vec{X})\|_2^2 - \|g(\vec{X})\|_2^2 \\
        &= 2 f(\vec{X})^T(g(\vec{X})-f(\vec{X})) + 2\eta^T(g(\vec{X})-f(\vec{X})) \\ 
        &\quad\quad\quad 
            2\xi(\vec{X})^T(g(\vec{X})-f(\vec{X})) + \|f(\vec{X})\|_2^2 - \|g(\vec{X})\|_2^2  \\
        &= - \|f(\vec{X})-g(\vec{X})\|_2^2 + 2\eta(\vec{X})^T(g(\vec{X})-f(\vec{X})) \\ 
        &\quad\quad\quad + 2\xi^T(g(\vec{X})-f(\vec{X})).
\end{align*} 
Observe that conditional on $\vec{X}$, by applying the Cauchy-Schwarz inequality, this quantity is normal with mean at most \[- \|f(\vec{X})-g(\vec{X})\|_2^2+ 2\|\eta(\vec{X})\|_2 \|f(\vec{X})-g(\vec{X})\|_2 \] and variance $4n\sigma^2 \|f-g\|_{L_2(\PP_n)}^2$.

We now derive estimates on $\|\eta(\vec{X})\|_2$ and $\|f-g\|_{L_2(\PP_n)}^2$ that hold with high probability. Suppose $n\|f-\bar{f}\|_{\LtwoP}^2<\delta^2$. From Lemma \ref{lemma:bernstein:nonparam:v2}, this condition along with our assumption $n\|f-g\|_{\LtwoP}^2\ge C^2\delta^2$ implies that with probability at least $1-\exp\left(-\frac{\delta^2}{20L}\right)- \exp\left(-\frac{\delta^2}{4L(C^2 + 1)}\right)$, we have  \begin{align} \label{eq:nonparam:testing_lemma:twoinequalities}
    n\|f-g\|_{L_2(\PP_n)}^2 \ge (C^2-1)\delta^2,\quad n\|f-\bar{f}\|_{L_2(\PP_n)}^2 \le 2\delta^2.
\end{align} Let $E$ denote the event that both inequalities of \eqref{eq:nonparam:testing_lemma:twoinequalities} hold, so \[\PP_{\bar{f}}(E^c) \le\exp\left(-\frac{\delta^2}{20L}\right) +\exp\left(-\frac{\delta^2}{4L(C^2 + 1)}\right).\]

The first inequality in \eqref{eq:nonparam:testing_lemma:twoinequalities} implies \[\|f(\vec{X})-g(\vec{X})\|_2 =\sqrt{n}\|f-g\|_{L_2(\PP_n)} \ge \delta\sqrt{C^2-1}\] in which case \[\delta \le \frac{\|f(\vec{X})-g(\vec{X})\|_2}{\sqrt{C^2-1}}.\] The other inequality simply states that $\|\eta(\vec{X})\|_2\le \delta\sqrt{2}$. Thus, the inequalities of \eqref{eq:nonparam:testing_lemma:twoinequalities} imply \[\|\eta(\vec{X})\|_2 \le \delta\sqrt{2} \le \frac{\sqrt{2}}{\sqrt{C^2-1}}\cdot \|f(\vec{X})-g(\vec{X})\|_2. \]

 Thus, conditional on $\vec{X}$, the event $E$ implies $\|\vec{Y}-f(\vec{X})\|_2^2 - \|\vec{Y}-g(\vec{X})\|_2^2$ is normal with mean bounded from above by
\begin{align*}
    \MoveEqLeft - \|f(\vec{X})-g(\vec{X})\|_2^2+ 2\|\eta(\vec{X})\|_2 \|f(\vec{X})-g(\vec{X})\|_2 \\ &\le \left(-1 +\frac{2\sqrt{2}}{\sqrt{C^2-1}}\right)\cdot\|f(\vec{X})-g(\vec{X})\|_2^2 \\
    &=  \underbrace{\left(-1 +\frac{2\sqrt{2}}{\sqrt{C^2-1}}\right)}_{<0}\cdot n\|f-g\|_{L_2(\PP_n)}^2   
\end{align*} and variance $4n\sigma^2 \|f-g\|_{L_2(\PP_n)}^2$. Note that $\frac{2\sqrt{2}}{\sqrt{C^2-1}}<1$ because $C>3$. 

Using a bound from \citet[Section 2.2.1]{van1996weak}, we note that for any $\rho<0<\tau$,
\begin{align*}
    \PP(\cN(\rho, \tau^2) \geq 0) \leq \exp(-\rho^2/(2\tau^2)). 
\end{align*} So conditional on $\vec{X}$ and the event $E$, $\psi = 1$ holds with probability bounded by 
\begin{align*}
\MoveEqLeft \exp\bigg(-\frac{\big(-1+\frac{2\sqrt{2}}{\sqrt{C^2-1}}\big)^2n^2 \|f-g\|_{L_2(\PP_n)}^4}{8n\sigma^2 \|f-g\|_{L_2(\PP_n)}^2}\bigg) \\ &=
    \exp\bigg(-\frac{\big(-1+\frac{2\sqrt{2}}{\sqrt{C^2-1}}\big)^2n 
        \|f-g\|_{L_2(\PP_n)}^2}{8\sigma^2} \bigg) \\
        &\le \exp\bigg(-\frac{\left(-1+\frac{2\sqrt{2}}{\sqrt{C^2-1}}\right)^2 (C^2-1)\delta^2}{8\sigma^2}\bigg) \\
        &= \exp\bigg(-\frac{\left(\sqrt{C^2-1}-2\sqrt{2}\right)^2 \delta^2}{8\sigma^2}\bigg).
\end{align*} Since $\PP_{\bar{f}}(\psi = 1|\vec{X},E)$ has an upper bound not depending on $\vec{X}$, it follows (by taking an expectation) that  \[\PP_{\bar{f}}(\psi = 1|E) \le \exp\bigg(-\frac{\left(\sqrt{C^2-1}-2\sqrt{2}\right)^2 \delta^2}{8\sigma^2}\bigg).\] Thus, \begin{align*}
    \MoveEqLeft\PP_{\bar{f}}(\psi = 1) \\ &= \PP_{\bar{f}}(\psi = 1|E)\cdot \PP_{\bar{f}}(E) + \PP_{\bar{f}}(\psi = 1|E^c)\cdot \PP_{\bar{f}}(E^c) \\
                 &\le \PP_{\bar{f}}(\psi = 1|E) + \PP_{\bar{f}}(E^c) \\
                 &\le \exp\left(-\frac{\left(\sqrt{C^2-1}-2\sqrt{2}\right)^2 \delta^2}{8\sigma^2}\right)+ \exp\left(-\frac{\delta^2}{20L}\right) \\ &\quad\quad+ \exp\left(-\frac{\delta^2}{4L(C^2 + 1)}\right).
\end{align*}    

If we instead assumed $n\|g-\bar{f}\|_{\LtwoP}^2<\delta^2$, a completely symmetric argument would result in a identical bound for $\PP_{\bar{f}}(\psi = 0)$. Thus, we conclude that \begin{align*}
    \MoveEqLeft \sup_{\bar{f}: n\|f-\bar{f}\|_{\LtwoP}^2<\delta^2}\PP_{\bar{f}}(\psi = 1) \vee \sup_{\bar{f}: n\|g-\bar{f}\|_{\LtwoP}^2<\delta^2}\PP_{\bar{f}}(\psi = 0) \\ &\le  \exp\left(-\frac{\left(\sqrt{C^2-1}-2\sqrt{2}\right)^2 \delta^2}{8\sigma^2}\right) + \exp\left(-\frac{\delta^2}{20L}\right) \\ &\quad\quad+\exp\left(-\frac{\delta^2}{4L(C^2 + 1)}\right).
\end{align*}

We define \[\Lambda(C,\sigma,L) = \min\left\{\frac{\left(\sqrt{C^2-1}-2\sqrt{2}\right)^2}{8\sigma^2}, \frac{1}{20L},\frac{1}{4L(C^2+1)} \right\},\] 
and replace the preceding upper bound by $3 \exp\left(-\Lambda(C,\sigma,L)\cdot\delta^2\right).$

\end{proof}

\begin{proof}[\hypertarget{proof:lemma:nonparam:intermediate_error}{Proof of Lemma \ref{lemma:nonparam:intermediate_error}}] For some $i\in[M]$, possibly distinct from $i^{\ast}$, we have $\|f_i - \bar{f}\|_{\LtwoP} \le \delta$ which means $n\|f_i - \bar{f}\|_{\LtwoP}^2 \le n\delta^2$. By the triangle inequality, we observe \begin{align*}
        \|f_{i^{\ast}} - \bar{f}\|_{\LtwoP} &\le  \|f_{i^{\ast}} - f_i\|_{\LtwoP} + \|f_i- \bar{f}\|_{\LtwoP} \\ &\le  \|f_{i^{\ast}} - f_i\|_{\LtwoP} +\delta.
    \end{align*} Note that  $\|\vec{Y}-f_i(\vec{X})\|_2^2 \ge \|\vec{Y}-f_{i^{\ast}}(\vec{X})\|_2^2$ since $f_{i^{\ast}}$ was defined to be the least squares estimator.
    
     Putting these facts together, \begin{align}
        \MoveEqLeft \PP_{\bar{f}}\big(\|f_{i^{\ast}}- \bar{f}\|_{\LtwoP} > (C +1)\delta \big) \le \PP_{\bar{f}}\left( \|f_{i^{\ast}}- f_i\|_{\LtwoP} > C\delta \right) \notag \\
        &\le \PP_{\bar{f}}\Big(\exists j:\|f_i - f_j\|_{\LtwoP} > C\delta \text{ and } \notag \\ &\qquad\qquad \|\vec{Y}-f_i(\vec{X})\|_2^2 \ge \|\vec{Y}-f_j(\vec{X})\|_2^2 \Big)\notag  \\
        &=\PP_{\bar{f}}\Big(\exists j:n\|f_i - f_j\|_{\LtwoP}^2 > C^2n\delta^2 \text{ and } \notag \\ &\qquad\qquad\|\vec{Y}-f_i(\vec{X})\|_2^2 \ge \|\vec{Y}-f_j(\vec{X})\|_2^2 \Big) \notag \\
        &\le \sum_j \PP_{\bar{f}}\Big(n\|f_i - f_j\|_{\LtwoP}^2 > C^2n\delta^2 \text{ and }\notag \\ &\qquad\qquad\|\vec{Y}-f_i(\vec{X})\|_2^2 \ge \|\vec{Y}-f_j(\vec{X})\|_2^2 \Big)  \notag \\
        &\le 3 M\exp\big(-\Lambda(C,\sigma, L)\cdot n\delta^2\big).\label{eq:lemma:nonparam:intermediate_error}
    \end{align} where we used Lemma \ref{lemma:testing:nonparam:gaussianbound} in \eqref{eq:lemma:nonparam:intermediate_error} with $\sqrt{n}\delta$ in place of $\delta$.        
    \end{proof}

\begin{proof}[\hypertarget{proof:subgaussian_remark_for_lemma_testing}{Proof of Remark \ref{subgaussian_remark_for_lemma_testing}}] In the proof of Lemma \ref{lemma:testing:nonparam:gaussianbound}, we first showed that \begin{align*}
 \|\vec{Y}-f(\vec{X})\|_2^2 - \|\vec{Y}-g(\vec{X})\|_2^2   &= - \|f(\vec{X})-g(\vec{X})\|_2^2 + 2\eta(\vec{X})^T(g(\vec{X})-f(\vec{X})) \\ 
        &\quad\quad\quad + 2\xi^T(g(\vec{X})-f(\vec{X})).
        \end{align*} Conditional on $\vec{X}$, this quantity is sub-Gaussian with mean bounded by \[- \|f(\vec{X})-g(\vec{X})\|_2^2+ 2\|\eta(\vec{X})\|_2 \|f(\vec{X})-g(\vec{X})\|_2. \] Moreover, the variance proxy can be shown to be $4n\sigma^2\|f-g\|_{L_2(\PP_n)}^2$ using the definition of a sub-Gaussian random variable. The proof of Lemma \ref{lemma:testing:nonparam:gaussianbound} then defines the event $E$ as when the inequalities in \eqref{eq:nonparam:testing_lemma:twoinequalities} hold, and $E$ satisfies \[\PP_{\bar{f}}(E^c) \le \exp\left(-\tfrac{\delta^2}{20L}\right)+ \exp\left(-\tfrac{\delta^2}{4L(C^2 + 1)}\right).\] Noting that Lemma \ref{lemma:bernstein:nonparam:v2} still holds with sub-Gaussian noise, we obtain as in Lemma \ref{lemma:testing:nonparam:gaussianbound} that conditional on $\vec{X}$, the event $E$ implies $\|\vec{Y}-f(\vec{X})\|_2^2 - \|\vec{Y}-g(\vec{X})\|_2^2$ is sub-Gaussian with mean bounded above by the negative quantity \[\left(-1+\tfrac{2\sqrt{2}}{\sqrt{C^2-1}}\right)\cdot n\|f-g\|_{L_2(\PP_n)}^2\] and variance proxy $4n\sigma^2\|f-g\|_{L_2(\PP_n)}^2$. A sub-Gaussian random variable $Z$ with mean $\rho$ and variance proxy $\tau^2$ can be shown to satisfy $P(Z>t+\rho)\le \exp(-t^2/2\tau^2)$ for any $t>0$ \citep[Section 2.3]{boucheron2013}, and if $\rho<0$, we can take $t = -\rho$ and obtain a bound of $P(Z>0)\le \exp(-\rho^2/2\tau^2)$. Applying this result, as in the previous lemma, we conclude that conditional on $\vec{X}$ and the event $E$, $\psi=1$ holds with probability bounded by \[\exp\left(-\tfrac{\left(\sqrt{C^2-1}-2\sqrt{2}\right)^2 \delta^2}{8\sigma^2}\right).\] The rest of the proof follows identically to Lemma \ref{lemma:testing:nonparam:gaussianbound}. Then we can use this generalized Lemma to obtain Lemmas \ref{lemma:nonparam:intermediate_error} and \ref{lemma:equivalent:norms} without any additional work, since other than using the bounds in Lemma \ref{lemma:testing:nonparam:gaussianbound}, normality is not required.
\end{proof}

    \begin{proof}[\hypertarget{proof:theorem:nonparam:upperbound:main}{Proof of Theorem \ref{theorem:nonparam:upperbound:main}}]

    If $J^{\ast}=1$, then $\varepsilon_{J^{\ast}}\asymp d$ and the risk bound trivially holds since $d$ is the $\LtwoP$ diameter. Suppose henceforth $J^{\ast}>1$.

    Define the event $A_j=\{\|\Upsilon_j-\bar f\|_{\LtwoP}>\tfrac{d}{2^{j-1}}\}$. Observe that $\PP(A_1)=0$ by definition of the diameter, and basic set inclusion properties imply $\PP(A_J) \le \PP(A_1) +\sum_{j=2}^J \PP(A_j\cap A_{j-1}^c)$. Let use now bound the probability of each of these intersections.
    
    \noindent\textsc{Part 1:} Bounding $\PP(A_j\cap A_{j-1}^c)$ for $j\ge 2$.

    We start with $j\ge 3$. Set $\delta = \tfrac{d}{2^{j-1}(C+1)}$. We claim the following holds: \begin{align}
         \MoveEqLeft \PP(A_j\cap A_{j-1}^c) \notag \\ &= \PP(\|\bar f -\Upsilon_j\|_{\LtwoP} > \tfrac{d}{2^{j-1}}, \|\bar f - \Upsilon_{j-1}\|_{\LtwoP}\le \tfrac{d}{2^{j-2}}) \notag \\ 
        &\le \sum_{g\in\cL(j-1)\cap B(\bar f, \tfrac{d}{2^{j-2}})}\PP(\|\bar f -\Upsilon_j\|_{\LtwoP} > \tfrac{d}{2^{j-1}},\Upsilon_{j-1}=g) \notag \\
        &= \sum_{g\in\cL(j-1)\cap B(\bar f, \tfrac{d}{2^{j-2}})}\PP(\|\bar f -f_{i^{\ast}}\|_{\LtwoP} > (C+1)\delta,\Upsilon_{j-1}=g) \notag \\
        &\le \sum_{g\in\cL(j-1)\cap B(\bar f, \tfrac{d}{2^{j-2}})}\PP(\|\bar f -f_{i^{\ast}}\|_{\LtwoP} > (C+1)\delta). \label{eq:A_j:and:A_j_minus_one:intermediate}
    \end{align} The second line applies a union bound: if
    $\|\bar f-\Upsilon_{j-1}\|_{\LtwoP} \le \tfrac{d}{2^{j-2}}$, then by definition of $\cL(j-1)$ as the $(j-1)$th level of our tree, $\Upsilon_{j-1}=g$ where $g\in\cL(j-1)\cap B(\bar f, \tfrac{d}{2^{j-2}})$. For the third line, we substitute the definition of $\delta$ and then recall $\Upsilon_j= f_{i^{\ast}}$ where $i^{\ast}=\argmin_{i\in [M]}\sum_{j=1}^n (Y_j-f_i(X_j))^2$ and the $\cO(g)=\{f_1,\dots,f_M\}$. We then drop the intersection with $\Upsilon_{j-1}=g$ for the final inequality.
    
    Now, using the definition of $\cO(g)$ and Lemma \ref{lemma:nonparam:pruned:tree:properties}, note that $f_1,\dots,f_M$ form a $\tfrac{d}{2^{j-2}c}=\tfrac{d}{2^{j-1}(C+1)}=\delta$-covering of the set $\cF'=B(g, \tfrac{d}{2^{j-2}})\cap \cF$ with cardinality of this set is bounded by $\cMFloc(\tfrac{d}{2^{j-2}},2c)$. Thus, we can apply Lemma \ref{lemma:nonparam:intermediate_error} to \eqref{eq:A_j:and:A_j_minus_one:intermediate}, noting requiring $c>8$ implies $C>3$ as required. Moreover, the cardinality of $\cL(j-1)\cap B(\bar f, \tfrac{d}{2^{j-2}})$ is also bounded by  $\cMFloc(\tfrac{d}{2^{j-2}},2c)$ by Lemma \ref{lemma:nonparam:bound:level:J:intersect:ball:mu}. Thus, continuing from \eqref{eq:A_j:and:A_j_minus_one:intermediate}, we have 
    \begin{align*}
         \PP(A_j\cap A_{j-1}^c) &\le \sum_{g\in\cL(j-1)\cap B(\bar f, \tfrac{d}{2^{j-2}})}3 \cMFloc(\tfrac{d}{2^{j-2}},2c)\exp\left(-\Lambda(C,\sigma,L)\cdot n\delta^2\right) \\
        &\le 3 \left[\cMFloc(\tfrac{d}{2^{j-2}},2c)\right]^2 \exp\left(-\Lambda(C,\sigma,L)\cdot n\delta^2\right) \\
        &= 3 \left[\cMFloc(\tfrac{c\varepsilon_j}{\sqrt{\Lambda(\tfrac{c}{2}-1,\sigma,L)}},2c)\right]^2 \exp\left(-n\varepsilon_j^2\right).
    \end{align*} The last step used our definition of $\varepsilon_j$, $\delta$, and $C$.

    Next, for $j=2$, note that $\PP(A_2\cap A_1^c)=\PP(A_2)$ since $\PP(A_1)=0$. Set $\delta = \tfrac{d}{2(C+1)}=\tfrac{d}{c}$. Well note that $\Upsilon_2$ is chosen from a maximal $\delta=\tfrac{d}{c}$-packing (hence covering) of $B(f_R,d)\cap\cF=B(\bar f, d)\cap \cF=\cF$. Such a packing has cardinality bounded by $\cMFloc(d, c)\le \cMFloc(d, 2c)$, using Lemma \ref{simple:lemma:monotone:nonparam}. Again applying Lemma \ref{lemma:nonparam:intermediate_error}, we have    \begin{align*} 
        \PP(A_2) &= \PP(\|\Upsilon_j-\bar{f}\|_{\LtwoP} > \delta(C+1)) \\ &\le 3\cMFloc(d, 2c)\exp(-\Lambda(C,\sigma,L)\cdot n\delta^2)  \\
        &= 3\cMFloc(\tfrac{c\varepsilon_2}{\sqrt{\Lambda(\tfrac{c}{2}-1,\sigma,L)}}, 2c)\exp(-n\varepsilon_2^2) \\
        &\le 3\left[\cMFloc(\tfrac{c\varepsilon_2}{\sqrt{\Lambda(\tfrac{c}{2}-1,\sigma,L)}}, 2c)\right]^2\exp(-n\varepsilon_2^2). 
    \end{align*}

    \noindent\textsc{Part 2:} Deducing a bound on $\PP(A_J)$ for any $1\le J \le J^{\ast}$.

    We now sum over our bounds using the $\PP(A_J) \le \sum_{j=2}^J \PP(A_j\cap A_{j-1}^c)$ fact, noting that by Lemma \ref{simple:lemma:monotone:nonparam}, \[\cMFloc(\tfrac{c\varepsilon_j}{\sqrt{\Lambda(\tfrac{c}{2}-1,\sigma,L)}}, 2c)\le \cMFloc(\tfrac{c\varepsilon_J}{\sqrt{\Lambda(\tfrac{c}{2}-1,\sigma,L)}}, 2c).\] Taking $a_J=\exp(-n\varepsilon_J^2)$, we obtain \begin{align*}
        \PP(A_J) &\le 3 \sum_{j=2}^J \left[\cMFloc(\tfrac{c\varepsilon_j}{\sqrt {\Lambda(\tfrac{c}{2}-1,\sigma,L)}}, 2c)\right]^2\exp(-n\varepsilon_j^2) \\
        &\le  3 \left[\cMFloc(\tfrac{c\varepsilon_J}{\sqrt {\Lambda(\tfrac{c}{2}-1,\sigma,L)}}, 2c)\right]^2\sum_{j=2}^J \exp(-n\varepsilon_j^2) \\
        &\le  3\left[\cMFloc(\tfrac{c\varepsilon_J}{\sqrt{\Lambda(\tfrac{c}{2}-1,\sigma,L)}}, 2c)\right]^2 \cdot \frac{a_J}{1-a_J}.
    \end{align*}
    Since $J^{\ast}>1$ and $1\le J\le J^{\ast}$, it follows \eqref{eq:upper_bound_condition:main_theorem} holds for $J$. Well, the first lower bound in \eqref{eq:upper_bound_condition:main_theorem} implies $\left[\cMFloc(\tfrac{c\varepsilon_J}{\sqrt{\Lambda(\tfrac{c}{2}-1,\sigma,L)}}, 2c)\right]^2 \le \exp(n\varepsilon_J^2/2)$ and the second lower bound implies $\tfrac{1}{1-a_J}\le 2$. Therefore, \begin{align*}
        \PP(A_J) \le 6 \exp(n\varepsilon_J^2/2)\cdot a_J = 6\exp(-n\varepsilon_J^2/2).
    \end{align*} 

    For the following parts, let $B_j$ be the event $\{\|\bar f-f^{\ast}\|_{\LtwoP}>\kappa \varepsilon_j\}$ where $\kappa= \tfrac{1+3c}{2\sqrt{\Lambda(c/2-1,\sigma,L)}}$ and $f^{\ast}$ is the output of precisely $J^{\ast}-1$ steps of the algorithm, i.e., $\Upsilon_{J^{\ast}}$. Recall $f^{\ast\ast}$ is the output of $J^{\ast\ast}\ge J^{\ast}$ steps, i.e., $\Upsilon_{J^{\ast\ast}+1}$.
    
    \noindent\textsc{Part 3:} Bounding $\PP(B_J)$ for integers $-\infty < J \le J^{\ast}$.

    First consider $1\le J\le J^{\ast}$. We can show $\PP(B_J)\le \PP(A_J)$ for such $J$ since if $B_J$ occurs, using the triangle inequality and Lemma \ref{lemma:nonparam:cauchy:sequence}, \begin{align*}
        \kappa \varepsilon_J &< \|\bar f-f^{\ast}\|_{\LtwoP} \\ &\le \|\bar f-\Upsilon_J\|_{\LtwoP} +\|\Upsilon_J-f^{\ast}\|_{\LtwoP} \\ &\le \|\bar f-\Upsilon_J\|_{\LtwoP} + \tfrac{d(2+4c)}{c2^J},
    \end{align*} which rearranging and substituting in $\kappa$ and $\varepsilon_J$ implies \begin{align}
        \|\bar f-\Upsilon_J\|_{\LtwoP} > \kappa \varepsilon_J - \tfrac{d(2+4c)}{c2^J} = \tfrac{d}{2^{J-1}}. \label{eq:temp:B_J:bound}
    \end{align} Thus $B_J$ implies $A_J$, yielding \begin{align*}
        \PP(B_J)\le \PP(A_J) \le  6\exp(-n\varepsilon_J^2/2).
    \end{align*}

    Noting then that $\kappa\varepsilon_J = \tfrac{(2+6c)d}{c2^J}> 6d$ for $J\le 0$, clearly $\PP(B_J)=0$ for all $J\le 0$. Thus our bound in \eqref{eq:temp:B_J:bound} remains valid all $-\infty<J\le J^{\ast}$.

    \noindent\textsc{Part 4:} Bounding $\PP(\|\bar f-f^{\ast\ast}\|_{\LtwoP}>\kappa' x)$ for $x\ge \varepsilon_{J^{\ast}}$ and $\kappa'=(2+\tfrac{1+2c}{1+3c})\kappa$.

    Pick any $x\ge \varepsilon_{J^{\ast}}$. Then for some integer $-\infty <J\le J^{\ast}$, $x\in[\varepsilon_{J},\varepsilon_{J-1})$ since $\bigcup_{J\le J^{\ast}}[\varepsilon_J,\varepsilon_{J-1}) = [\varepsilon_{J^{\ast}},\infty)$. We have by Lemma \ref{lemma:nonparam:cauchy:sequence} and definitions of $\kappa$, $\varepsilon_J$, $f^{\ast}$, and $f^{\ast\ast}$ that \begin{align*}
        \|\bar f-f^{\ast\ast}\|_{\LtwoP} &\le \|\bar f - f^{\ast}\|_{\LtwoP} + \|f^{\ast}-f^{\ast\ast}\|_{\LtwoP} \\
        &= \|\bar f - f^{\ast}\|_{\LtwoP} + \|\Upsilon_{J^{\ast}}- \Upsilon_{J^{\ast\ast}+1}\|_{\LtwoP} \\ &\le \|\bar f - f^{\ast}\|_{\LtwoP} + \tfrac{d(2+4c)}{c2^{J^{\ast}}} \\
        &= \|\bar f - f^{\ast}\|_{\LtwoP} + \tfrac{1+2c}{1+3c}\kappa\varepsilon_J^{\ast}\\
        &\le  \|\bar f - f^{\ast}\|_{\LtwoP} + \tfrac{1+2c}{1+3c}\kappa x.
    \end{align*} Consequently, recalling the definition of $\kappa'$ and noting that $2\kappa x \ge \kappa \varepsilon_{J-1}$ since $x\ge\varepsilon_J=\varepsilon_{J-1}/2$, we obtain \begin{align*}
        \PP(\|\bar f-f^{\ast\ast}\|_{\LtwoP}>\kappa' x) &\le \PP( \|\bar f - f^{\ast}\|_{\LtwoP} + \tfrac{1+2c}{1+3c}\kappa x > \kappa'x) \\
        &= \PP( \|\bar f - f^{\ast}\|_{\LtwoP}  > 2\kappa x) \\ 
        &\le \PP( \|\bar f - f^{\ast}\|_{\LtwoP}  > \kappa \varepsilon_{J-1}) \\
        &= \PP(B_{J-1}) \\
        &\le 6\exp(-n\varepsilon_{J-1}^2/2) \\
        &\le 6\exp(-nx^2/2).
    \end{align*} The last few lines used the definition of the event $B_{J-1}$ and then that $x< \varepsilon_{J-1}.$

    \noindent\textsc{Part 5:} Bounding $\EE_{\vec{X},\vec{Y}}\|\bar{f} - f^{\ast\ast}(\vec{X},\vec{Y})\|_{\LtwoP}^2.$ 

    We integrate the bound from Part 4: \begin{align*}
       \MoveEqLeft \EE_{\vec{X},\vec{Y}}\|\bar{f} - f^{\ast\ast}(\vec{X},\vec{Y})\|_{\LtwoP}^2 \\ &= \int_0^{\infty} \PP(\|\bar f-f^{\ast\ast}\|_{\LtwoP}^2> x)\mathrm{d}x \\
        &= 2{\kappa'}^2\int_0^{\infty} t\cdot \PP(\|\bar f-f^{\ast\ast}\|_{\LtwoP}> \kappa't)\mathrm{d}t \\
        &\le 2{\kappa'}^2\int_0^{\varepsilon_J^{\ast}} t \mathrm{d}t + 2{\kappa'}^2\int_{\varepsilon_{J^{\ast}}}^{\infty}  t\cdot \PP(\|\bar f-f^{\ast\ast}\|_{\LtwoP}> \kappa't)\mathrm{d}t \\
        &\le {\kappa'}^2 \varepsilon_{J^{\ast}}^2 +12{\kappa'}^2\int_{\varepsilon_{J^{\ast}}}^{\infty}t\cdot\exp(-nt^2/2)\mathrm{d}t \\
        &= {\kappa'}^2 \varepsilon_{J^{\ast}}^2 +12{\kappa'}^2\cdot n^{-1}\exp(-n\varepsilon_{J^{\ast}}^2/2).
    \end{align*}
        
        Our assumption \eqref{eq:upper_bound_condition:main_theorem}  implies $n^{-1} \le \varepsilon_{J^{\ast}}^2/\log 2$, while $\exp(-n\eta_{J^{\ast}}^2/2)\le 1$, so the entire bound bounded by $\varepsilon_{J^{\ast}}^2$ up to constants depending on $c$, $\sigma$, and $L$.
    \end{proof}

    \begin{proof}[\hypertarget{proof:theorem:bounded:minimax:achieved}{Proof of Theorem \ref{theorem:bounded:minimax:achieved}}]
        The $\varepsilon^{\ast}=0$ edge case implies $d=0$ and we trivially achieve the minimax rate of $0$. Assume henceforth $\varepsilon^{\ast}>0$.
        
        \textsc{Case 1:} Suppose $n{\varepsilon^{\ast}}^2 >8\log 2$.
         Set $\alpha = \min(\sigma/2,1)\in(0,1]$ and $\delta^{\ast}=\alpha \varepsilon^{\ast}$. Note that $\alpha^2 \le \frac{\sigma^2}{4}$. Then using the monotonicity of  $\cMFloc(\cdot, c)$, we obtain \begin{align*}
            \log \cMFloc(\delta^{\ast}, c) &\ge \lim_{t\downarrow 0}\log \cMFloc(\varepsilon^{\ast}-t, c) \ge \lim_{t\downarrow 0}n({\varepsilon^{\ast}}-t)^2  = \frac{n{\varepsilon^{\ast}}^2}{2} +\frac{n{\varepsilon^{\ast}}^2}{2} \\ &> \frac{n{\delta^{\ast}}^2}{2\alpha^2} +4\log 2 \ge 4\cdot\frac{n{\delta^{\ast}}^2}{2\sigma^2} + 4\log2 \ge 4\left( \frac{n{\delta^{\ast}}^2}{2\sigma^2} \vee \log 2\right).
        \end{align*} Thus, the lower bound condition from Lemma \ref{lemma:nonparam:lowerbound} is satisfied for $\delta^{\ast}$, so indeed ${\varepsilon^{\ast}}^2$ is a lower bound for the minimax rate up to absolute constants.

        We now show ${\varepsilon^{\ast}}^2$ is an upper bound. Following \cite[Theorem 12]{lecamshamindramatey2022}, we set $\eta = \frac{c}{\sqrt{\Lambda(c/4-1,\sigma,L)}}\wedge 1\in(0,1)$, pick $D$ sufficiently large such that $D\eta>1$, and set $\delta =2D\varepsilon^{\ast}$. Again using the definition of $\varepsilon^{\ast}$ as a supremum and monotonicity of $\cMFloc(\cdot, c)$, we have \begin{align*}
            n\delta^2 &= 4D^2 n{\varepsilon^{\ast}}^2 = \frac{4}{\eta^2} \cdot n  (D\eta \cdot \varepsilon^{\ast})^2 > \frac{4}{\eta^2}\cdot \log \cMFloc(D\eta\varepsilon^{\ast}, c) \\ &\ge 4\log \cMFloc(D\eta\varepsilon^{\ast}, c) 
            \ge 4\log \cMFloc\left(\frac{c D\varepsilon^{\ast}}{\sqrt{\Lambda(c/4-1,\sigma,L)}}, c\right) \\ &= 4\log \cMFloc\left(\frac{c\cdot\delta}{2\cdot\sqrt{\Lambda(c/4-1,\sigma,L)}}, c\right).
        \end{align*} Further, since $n{\varepsilon^{\ast}}^2 > 8 \log 2$ and clearly $D>1$, we have $n\delta^2 = 4n D^2 {\varepsilon^{\ast}}^2> 32D^2\log 2>\log 2$, so indeed \[n\delta^2 > 2\log \left[\cMFloc\left(\frac{c\cdot\delta}{2\sqrt{\Lambda(c/4-1,\sigma,L)}}, c\right)\right]^2\vee \log 2.\]
 In other words, $\delta$ satisfies the condition in \eqref{eq:upper_bound_condition:main_theorem}, noting by Remark \ref{remark:swapping:2c:for:c} we may swap out $2c$ with $c$.
 
 Now consider the non-decreasing map \[\psi\colon 0<x\mapsto nx^2 -  2\log \left[\cMFloc\left(x\cdot\frac{ c}{2\sqrt{\Lambda(c/4-1,\sigma,L)}}, c\right)\right]^2\vee \log 2.\] We have shown that $\psi(\delta)>0$. Recall our definition of $\varepsilon_{J^{\ast}}$ from Theorem \ref{theorem:nonparam:upperbound:main}, and let us  show that $\delta \ge  \varepsilon_{J^{\ast}}/2$. Suppose no $J\ge 1$ exists such that $\varepsilon_J$ satisfies \eqref{eq:upper_bound_condition:main_theorem}, in which case $J^{\ast}=1$. Then $\psi(\varepsilon_{J^{\ast}}/2)\le \psi(\varepsilon_{J^{\ast}})<0<\psi(\delta)$ so $\delta \ge \varepsilon_{J^{\ast}}/2$. On the other hand, suppose some $J$ does exist such that $\varepsilon_J$ satisfies \eqref{eq:upper_bound_condition:main_theorem}.  Because $J^{\ast}<\infty$ is chosen maximally in our algorithm, $\varepsilon_{J^{\ast}+1}=\varepsilon_{J^{\ast}}/2$ does not satisfy \eqref{eq:upper_bound_condition:main_theorem}, so we have $\psi(\varepsilon_{J^{\ast}}/2) = \psi(\varepsilon_{J^{\ast}+1}) <0 < \psi(\delta)$, so $\delta \ge \varepsilon_{J^{\ast}}/2$. 
 
Since Theorem  \ref{theorem:nonparam:upperbound:main} upper bounds the minimax rate by $\varepsilon_{J^{\ast}}^2$ and we just showed $\delta\gtrsim \varepsilon_{J^{\ast}}$, the rate is also upper bounded by $\delta^2\asymp {\varepsilon^{\ast}}^2$. Thus, the minimax has a lower and upper bound of rate ${\varepsilon^{\ast}}^2$ up to absolute constants. But if ${\varepsilon^{\ast}}^2$ is a lower bound, then so is ${\varepsilon^{\ast}}^2\wedge d^2$; further, the squared diameter is always an upper bound for the minimax rate, so  ${\varepsilon^{\ast}}^2\wedge d^2$ is also an upper bound up to absolute constants. Thus, our minimax rate is ${\varepsilon^{\ast}}^2\wedge d^2$ up to absolute constants.

    \textsc{Case 2:} Suppose  $n{\varepsilon^{\ast}}^2\le 8\log 2$.  
    Observe by definition of $\varepsilon^{\ast}$ as a supremum that \begin{align} \label{eq:8log2:metric:entropy:bound}
        \log \cMFloc(2\varepsilon^{\ast}, c) \le 4n {\varepsilon^{\ast}}^2 \le 32\log 2.
    \end{align}

    Now repeating the logic of \citet[Lemma 1.3]{prasadan2024informationtheoreticlimitsrobust} for star-shaped sets, there exists a line segment in $\cF$ of length $d/3$ in the $\LtwoP$ distance.
    Form a ball $B$ of radius $2\varepsilon^{\ast}$ centered at the midpoint of this line segment. Consider the diameter of $B$ that coincides with the chosen line segment in $\cF$ at infinitely many points. Since an interval of length $4\varepsilon^{\ast}$ if partitioned into subintervals of length less than $(2\varepsilon^{\ast}/c)$ is composed of at least $2c$ subintervals, for sufficiently large $c$, at  least $\exp(32\log 2)$ many points at least $(2\varepsilon^{\ast}/c)$ distance apart can be chosen along this particular diameter of $B$. All of these points cannot also belong to $\cF$, otherwise we will have formed a $(2\varepsilon^{\ast}/c)$-packing of $B\cap\cF$ with $\exp(32\log 2)$ points, contradicting $\log \cMFloc(2\varepsilon^{\ast}, c)  \le 32\log 2$. This means the line segment of $\cF$ we picked is strictly contained in the diameter of $B$ we picked, so that $d/3\le 4\varepsilon^{\ast} \le 8\sqrt{2n^{-1}\log 2}$, using \eqref{eq:8log2:metric:entropy:bound}.

    Now, consider an $\varepsilon$ which is proportional to $d$. In this case, $n\varepsilon^2$ is upper bounded by some constant since $d \le 24\sqrt{2n^{-1}\log 2}$, while $\log \cMFloc(2\varepsilon, c)$ can be made arbitrarily large by increasing $c$, to ensure the condition from Lemma \ref{lemma:nonparam:lowerbound} for the lower bound is attained. So $\varepsilon^2$ is a lower bound for the minimax-rate up to constants, and thus by proportionality, $d^2$ is a lower bound up to absolute constants. Thus ${\varepsilon^{\ast}}^2\wedge d^2$ is a lower bound for the minimax rate.
    
    But $d^2$ is also always an upper bound for the minimax rate, so the minimax rate is of order $d^2$ up to constants. But in this second case where $n{\varepsilon^{\ast}}^2\le 8\log 2$, we have $d\le 12\varepsilon^{\ast}$, showing that ${\varepsilon^{\ast}}^2$ and thus ${\varepsilon^{\ast}}^2\wedge d^2$ is also an upper bound up to absolute constants. We conclude therefore the minimax rate is ${\varepsilon^{\ast}}^2\wedge d^2$ up to constants. 
   \end{proof}

\section{Proofs for Section \ref{section:nonparam:adaptive}}

\begin{proof}[\hypertarget{proof:theorem:nonparam:upperbound:main:adaptive}{Proof of Theorem \ref{theorem:nonparam:upperbound:main:adaptive}}]

Set $A_j=\{\|\Upsilon_j-\bar f\|_{\LtwoP}>\tfrac{d}{2^{j-1}}\}$. We write $f^{\dagger}=\Upsilon_{J^{\dagger}}$ to be the output of exactly $J^{\dagger}-1$ steps, while we write $f^{\ddagger}=\Upsilon_{J^{\ddagger}+1}$ as the output of $J^{\ddagger}\ge J^{\dagger}$ steps.

If $J^{\ast}=1$, then $\varepsilon_{J^{\ast}}\asymp d$ and trivially $\EE_{\vec{X},\vec{Y}}\|\bar{f} - f^{\ddagger}(\vec{X},\vec{Y})\|_{\LtwoP}^2\lesssim d^2$. Suppose $J^{\ast}>1$.

\noindent\textsc{Part 1:} Bounding $\PP(A_j\cap A_{j-1}^c)$ for $j\ge 2$.

Pick $j\ge 3$ and repeat the logic leading up to \eqref{eq:A_j:and:A_j_minus_one:intermediate} in the proof of Theorem \ref{theorem:nonparam:upperbound:main} to obtain \begin{align*}
          \PP(A_j\cap A_{j-1}^c) &\le \sum_{g\in\cL(j-1)\cap B(\bar f, \tfrac{d}{2^{j-2}})}\PP(\|\bar f -f_{i^{\ast}}\|_{\LtwoP} > (C+1)\delta),
    \end{align*} where $\delta = \tfrac{d}{2^{j-1}(C+1)}$. Recall  $\Upsilon_j= f_{i^{\ast}}$ where $i^{\ast}=\argmin_{i\in [M]}\sum_{j=1}^n (Y_j-f_i(X_j))^2$ for $\cO(g)=\{f_1,\dots,f_M\}$ when $\Upsilon_{j-1}=g$.

    Now,  $\cO(g)$ was originally formed as a (maximal) $\tfrac{d}{2^{j-1}c}$-packing of $B(g, \tfrac{d}{2^{j-2}})\cap\cF$. We then pruned the packing but this does not affect that it is a packing. Thus, $|\cO(g)|\le \cMadloc(g, \tfrac{d}{2^{j-2}}, 2c)$. But since $\|g-\bar f\|_{\LtwoP}\le \tfrac{d}{2^{j-2}}$, we actually have $|\cO(g)|\le \cMadloc(\bar f, \tfrac{d}{2^{j-3}}, 4c)$ by Lemma \ref{lemma:adloc:technical}.
    Moreover, Lemma \ref{lemma:nonparam:pruned:tree:properties} states that $\cO(g)$ forms a $\tfrac{d}{2^{j-2}c}=\tfrac{d}{2^{j-1}(C+1)}=\delta$-covering of $\cF'=B(g, \tfrac{d}{2^{j-2}})\cap \cF$. Hence by Lemma \ref{lemma:nonparam:intermediate_error}, 
    \begin{align*}
          \MoveEqLeft \PP(A_j\cap A_{j-1}^c) \\ &\le \sum_{g\in\cL(j-1)\cap B(\bar f, \tfrac{d}{2^{j-2}})}3 |\cO(g)| \exp(-\Lambda(C, \sigma, L) n\delta^2) \\
          &\le  \sum_{g\in\cL(j-1)\cap B(\bar f, \tfrac{d}{2^{j-2}})}3 \cdot \cMadloc(\bar f, \tfrac{d}{2^{j-3}}, 4c)\exp(-\Lambda(C, \sigma, L) n\delta^2). 
    \end{align*} Note once more that assuming $c>8$ implies $C>3$ so the application of the lemma was valid.
    
    Next, since $\cL(j-1)$ forms a $\tfrac{d}{2^{j-1}c}$-packing of $\cF$ by the same lemma, it follows that the number of choices for $g\in\cL(j-1)\cap B(\bar f, \tfrac{d}{2^{j-2}})$ is bounded by $\cMadloc(\bar f, \tfrac{d}{2^{j-2}}, 2c)$. Thus,  \begin{align*}
            \MoveEqLeft \PP(A_j\cap A_{j-1}^c) \\ &\le  3 \cdot \cMadloc(\bar f, \tfrac{d}{2^{j-2}}, 2c)\cdot \cMadloc(\bar f, \tfrac{d}{2^{j-3}}, 4c)\exp(-\Lambda(C, \sigma, L) n\delta^2) \\
          &\le 3 \cdot \left[\cMadloc(\bar f, \tfrac{d}{2^{j-2}}, 4c)\right]^2\exp(-\Lambda(C, \sigma, L) n\delta^2)\\
          &= 3 \cdot \left[\cMadloc(\bar f,\tfrac{c\varepsilon_j}{\sqrt{\Lambda(\tfrac{c}{2}-1, \sigma, L)}},4c)\right]^2\exp(-\Lambda(\tfrac{c}{2}-1, \sigma, L) n\delta^2)
    \end{align*} using non-increasing properties from Lemma \ref{lemma:monotone:nonparam:adaptive}.

    For $j=2$, recall $\PP(A_2\cap A_1^c)=\PP(A_2)$ since $\PP(A_1)=0$. Well $\Upsilon_2=f_{i^{\ast}}$ was picked from a maximal $d/c$-packing (hence covering) of $B(f_R,d)\cap\cF$, without pruning. But $B(f_R,d)\cap\cF=B(\bar{f}, d)\cap\cF$, so the cardinality is bounded by $\cMadloc(\bar f,d, c)\le \cMadloc(\bar f,d, 4c) $. Hence applying Lemma \ref{lemma:nonparam:intermediate_error} with $\delta = \tfrac{d}{c}=\tfrac{d}{2(C+1)}$, \begin{align*}
        \PP(A_2) &\le \PP(\|\bar f -f_{i^{\ast}}\|_{\LtwoP} > d/2) \\
        &= \PP(\|\bar f -f_{i^{\ast}}\|_{\LtwoP} > (C+1)\delta) \\
        &\le 3 |\cO(f_R)| \exp(-\Lambda(C, \sigma, L) n\delta^2)  \\
        &\le 3 \cdot \cMadloc(\bar f,d, 4c)\exp(-\Lambda(C, \sigma, L) n\delta^2)\\
        &\le 3 \cdot \left[\cMadloc(\bar f, d, 4c)\right]^2\exp(-\Lambda(C, \sigma, L) n\delta^2) \\
        &= 3 \cdot \left[\cMadloc(\bar f, \tfrac{c\varepsilon_2}{\sqrt{\Lambda(\tfrac{c}{2}-1, \sigma, L)}},4c)\right]^2\exp(-n\varepsilon_2^2).
    \end{align*} 
    
    \noindent\textsc{Part 2:} Bounding $\PP(A_J)$ for $1\le J\le J^{\ast}$.

     Set $a_J=\exp(-n\varepsilon_J^2)$. Summing over $2\le j\le J$ and using Lemma \ref{lemma:monotone:nonparam:adaptive} to replace $\varepsilon_j$ with $\varepsilon_J$ in the local metric entropy, we conclude as in the non-adaptive case that \begin{align*}
        \PP(A_J) &\le \sum_{j=2}^J \PP(A_j\cap A_{j-1}^c) \le 3\sum_{j=2}^J \left[\cMadloc(\bar f, \tfrac{c\varepsilon_j}{\sqrt{\Lambda(C, \sigma, L)}},4c)\right]^2\exp(- n\varepsilon_j^2) \\
        &\le 3 \left[\cMadloc(\bar f, \tfrac{c\varepsilon_J}{\sqrt{\Lambda(C, \sigma, L)}},4c)\right]^2\cdot \frac{a_J}{1-a_J}.
    \end{align*} Since $J^{\ast}>1$ and $1\le J\le J^{\ast}$, \eqref{eq:upper_bound_condition:main_theorem:adaptive} implies both that \[\bigg[\cMadloc(\bar f, \tfrac{c\varepsilon_J}{\sqrt{\Lambda(C, \sigma, L)}},4c)\bigg]^2\le \exp(n\varepsilon_J/2)\] and $\tfrac{1}{1-a_J}\le 2$. Hence \begin{align*}
        \PP(A_J) &\le 6\exp(n\varepsilon_J^2/2)\cdot a_J = 6\exp(-n\varepsilon_J^2/2).
    \end{align*}

    Next, recall we set $f^{\dagger}=\Upsilon_{J^{\dagger}}$ as the output of exactly $J^{\dagger}-1$ steps. Define $B_j=\{\|\bar f-f^{\dagger}\|_{\LtwoP}>\kappa \varepsilon_j\}$ where $\kappa= \tfrac{1+3c}{2\sqrt{\Lambda(c/2-1,\sigma,L)}}.$
    
 \noindent\textsc{Part 3:} Bounding $\PP(B_J)$ for integers $-\infty< J\le J^{\ast}$.

     Well, let us first show $\PP(B_J)\le \PP(A_J)$ for $1\le J\le J^{\ast}$. Suppose $B_J$ holds. By comparing \eqref{eq:upper_bound_condition:main_theorem:adaptive:preliminary} with \eqref{eq:upper_bound_condition:main_theorem:adaptive}, we see that $J^{\ast}\le J^{\dagger}$. Hence $J\le J^{\dagger}$. Thus, the triangle inequality and Lemma \ref{lemma:nonparam:cauchy:sequence} applied with $\Upsilon_J$ and $f^{\dagger}=\Upsilon_{J^{\dagger}}$ imply 
     \begin{align*}
            \kappa \varepsilon_J &< \|\bar f-f^{\dagger}\|_{\LtwoP} \\ &\le \|\bar f-\Upsilon_J\|_{\LtwoP} +\|\Upsilon_J-f^{\dagger}\|_{\LtwoP} \\
            &\le \|\bar f-\Upsilon_J\|_{\LtwoP} + \tfrac{d(2+4c)}{c2^J}.
        \end{align*}
        Then we again conclude $A_J$ holds since \begin{align}
            \|\bar f-\Upsilon_J\|_{\LtwoP} > \kappa \varepsilon_J - \tfrac{d(2+4c)}{c2^J} = \tfrac{d}{2^{J-1}}. \label{eq:temp:B_J:bound:adaptive}
        \end{align} In other words $B_J\subseteq A_J$, so \[\PP(B_J)\le \PP(A_J)\le 6\exp(-n\varepsilon_J^2/2),\] since $J\le J^{\ast}$ and we 
     can also  apply the bound from Part 2. Since $\PP(B_J)=0$ for all $J\le 0$ due to $\kappa\varepsilon_J > d$ for such $J$, our bound in \eqref{eq:temp:B_J:bound:adaptive} holds for all integers $-\infty< J \le J^{\ast}$.

 \noindent\textsc{Part 4:} Bounding $\PP(\|\bar f - f^{\ddagger}\|_{\LtwoP}>\kappa'x)$ for all $x\ge\varepsilon_{J^{\ast}}$, where $\kappa' = (2+\tfrac{1+2c}{1+3c})\kappa$.

     We now re-use the logic from Part 4 of the proof of the non-adaptive Theorem \ref{theorem:nonparam:upperbound:main} but using $f^{\dagger}$ and $f^{\ddagger}$ in place of $f^{\ast}$ and $f^{\ast\ast}$, respectively. Recall $f^{\ddagger}=\Upsilon_{J^{\ddagger}+1}$ is the output of $J^{\ddagger}\ge J^{\dagger}$ steps. Pick any $x\ge\varepsilon_{J^{\ast}}$ so for some $0<J\le J^{\ast}$ we have $x\in[\varepsilon_J,\varepsilon_{J-1})$. Then we have
      \begin{align*}
            \|\bar f-f^{\ddagger}\|_{\LtwoP} &\le \|\bar f - f^{\dagger}\|_{\LtwoP} + \|f^{\dagger}-f^{\ddagger}\|_{\LtwoP} \\
            &= \|\bar f - f^{\dagger}\|_{\LtwoP} + \|\Upsilon_{J^{\ddagger}+1}- \Upsilon_{J^{\dagger}}\|_{\LtwoP} \\ &\le \|\bar f - f^{\dagger}\|_{\LtwoP} + \tfrac{d(2+4c)}{c2^{J^{\dagger}}} \\
            &= \|\bar f - f^{\dagger}\|_{\LtwoP} + \tfrac{1+2c}{1+3c}\kappa\varepsilon_{J^{\dagger}} \\
            &\le \|\bar f - f^{\dagger}\|_{\LtwoP} + \tfrac{1+2c}{1+3c} \kappa x.
        \end{align*} The last line used that $x\ge \varepsilon_{J^{\ast}} \ge \varepsilon_{J^{\dagger}}$, which in turn follows from $J^{\ast}\le J^{\dagger}$. Using the definition $\kappa'$ and  that $2\kappa x \ge \kappa \varepsilon_{J-1}$ since $x\ge\varepsilon_J=\varepsilon_{J-1}/2$, we have \begin{align*}
            \PP(\|\bar f-f^{\ddagger}\|_{\LtwoP}>\kappa' x) &\le \PP( \|\bar f - f^{\dagger}\|_{\LtwoP} + \tfrac{1+2c}{1+3c}\kappa x> \kappa'x) \\
            &= \PP( \|\bar f - f^{\dagger}\|_{\LtwoP}  > 2\kappa x) \\ 
            &\le \PP( \|\bar f - f^{\dagger}\|_{\LtwoP}  > \kappa \varepsilon_{J-1}) \\
            &= \PP(B_{J-1}) \\
            &\le 6\exp(-n\varepsilon_{J-1}^2/2) \\
            &\le 6\exp(-nx^2/2).
        \end{align*} The last few lines used Part 3's bound on $\PP(B_J)$ and that $X<\varepsilon_{J-1}$.

\noindent\textsc{Part 5:} Bounding $\EE_{\vec{X},\vec{Y}}\|\bar f - f^{\ddagger}(\vec{X},\vec{Y})\|_{\LtwoP}^2$.

    We integrate Part 4's bound: \begin{align*}
       \MoveEqLeft \EE_{\vec{X},\vec{Y}}\|\bar{f} - f^{\ddagger}(\vec{X},\vec{Y})\|_{\LtwoP}^2 \\ &= \int_0^{\infty} \PP(\|\bar f-f^{\ddagger}\|_{\LtwoP}^2> x)\mathrm{d}x \\
        &= 2{\kappa'}^2\int_0^{\infty} t\cdot \PP(\|\bar f-f^{\ddagger}\|_{\LtwoP}> \kappa't)\mathrm{d}t \\
        &\le 2{\kappa'}^2\int_0^{\varepsilon_J^{\ast}} t \mathrm{d}t + 2{\kappa'}^2\int_{\varepsilon_{J^{\ast}}}^{\infty}  t\cdot \PP(\|\bar f-f^{\ddagger}\|_{\LtwoP}> \kappa't)\mathrm{d}t \\
        &\le {\kappa'}^2 \varepsilon_{J^{\ast}}^2 +12{\kappa'}^2\int_{\varepsilon_{J^{\ast}}}^{\infty}t\cdot\exp(-nt^2/2)\mathrm{d}t \\
        &= {\kappa'}^2 \varepsilon_{J^{\ast}}^2 +12{\kappa'}^2\cdot n^{-1}\exp(-n\varepsilon_{J^{\ast}}^2/2).
    \end{align*} Since \eqref{eq:upper_bound_condition:main_theorem} implies  $n^{-1}< \varepsilon_{J^{\ast}}^2/\log 2$ as $J^{\ast}>1$ and the exponential term is always $\le 1$, the entire expression is bounded by $\varepsilon_{J^{\ast}}^2$ up to constants.
    \end{proof}

\end{document}